\documentclass[a4paper,reqno]{amsart}

\usepackage{graphicx}
\usepackage{mathrsfs}
\usepackage{amsfonts}
\usepackage{amssymb}
\usepackage{amsmath}
\usepackage{amsthm}
\usepackage{amscd}
\usepackage[all,2cell]{xy}
\usepackage{color}
\usepackage[pagebackref,colorlinks]{hyperref}
\usepackage{enumerate}

\usepackage{comment}
\usepackage{hyperref}
\UseAllTwocells \SilentMatrices

\numberwithin{equation}{section}

\usepackage{comment}

\theoremstyle{plain}
\newtheorem{thm}{Theorem}[section]
\newtheorem{prop}[thm]{Proposition}
\newtheorem{cor}[thm]{Corollary}
\newtheorem{lemma}[thm]{Lemma}

\theoremstyle{definition}
\newtheorem{deff}[thm]{Definition}
\newtheorem{example}[thm]{Example}
\theoremstyle{remark}
\newtheorem{rmk}[thm]{\bf Remark}

\def\g{\gamma}
\def\d{\delta}
\def\G{\Gamma}
\def\mg{\mathcal{G}}
\def\xra{\xrightarrow[]{}}
\def\a{\alpha}
\def\b{\beta}

\def\v{\varepsilon}

\def\O{\mathcal{O}}
\def\I{\mathcal{I}}

\def\SS{\mathcal{S}}

\def\sub{\subseteq}

\def\F{\mathbb{F}}

\def\star{{\rm Star}}

\def \Z{\mathbb Z}
\def\-{\text{-}}

\def\LL{\mathcal{L}}

\newcommand{\supp}{\operatorname{supp}}
\newcommand{\Iso}{\operatorname{Iso}}

\newcommand{\Ker}{\operatorname{Ker}}

\newcommand{\gr}{\operatorname{gr}}

\newcommand{\id}{\operatorname{id}}

\newcommand{\note}[1]{\textcolor{blue}{$\Longrightarrow$ #1}}

\begin{document}

\title{Graded Steinberg algebras and partial actions}

\author{Roozbeh Hazrat}
\address{
Western Sydney University\\
Australia} \email{r.hazrat@westernsydney.edu.au, h.li@westernsydney.edu.au}

\author{Huanhuan Li}



\date{\today}

\begin{abstract} Given a graded ample Hausdorff groupoid, we realise its graded Steinberg algebra as a partial skew inverse semigroup ring.  We use this to show that for a partial action of a discrete group on a locally compact Hausdorff topological space,  the Steinberg algebra of the associated groupoid is graded isomorphic to the corresponding partial skew group ring. We show that there is a one-to-one correspondence between the open invariant subsets of the topological space and the graded ideals of the partial skew group ring. We also consider the algebraic version of the partial $C^*$-algebra of an abelian group and realise it as a partial skew group ring via a partial action of the group on a topological space. Applications to the theory of Leavitt path algebras are given. 
\end{abstract}

\maketitle

\begin{center}
\emph{In memory of Maryam Mirzakhani (1977-2017)}
\end{center}

\section{Introduction}

The notion of crossed product by a partial action has its origin in the concept of crossed product by a partial automorphism introduced by Exel in \cite{exel1994}. Crossed products of $C^*$-algebras by partial actions of discrete groups were defined in \cite{maclanahan} by McClanahan. Skew group rings were introduced by Dokuchaev and Exel in \cite{de} as algebraic analogues of $C^*$-crossed products by partial actions. The latter algebras are a powerful tool in the study of operator algebras (see \cite{exel1994, exel1997, exel1998, quiggraeburn}), and so it is important to realise $C^*$-algebras as partial crossed products (see \cite{be,el} for example), as one can then benefit from the established theory about partial crossed products.

Sieben \cite{sieben} introduced the notion of a crossed product by an action of an inverse semigroup on a $C^*$-algebra using covariant representations. Later, a definition of crossed product for actions of inverse semigroups on $C^*$-algebras, without resorting to covariant representations was presented in \cite{exelvieira}. The algebraic version for actions of inverse semigroups on algebras were investigated \cite{bussexel}.

Recently Steinberg algebras were introduced  in~\cite{cfst,st} as an algebraisation of the groupoid
$C^*$-algebras first studied by Renault~\cite{re}. Steinberg algebras include
Leavitt and Kumjian--Pask algebras as well as inverse semigroup algebras (see \cite{cp,cs,st}). These classes of algebras have been attracting significant attention, with particular interest in the graded ideal structures of these algebras. 

In this note we relate these two class of algebras. Starting from a graded ample Hausdorff groupoid $\mg$ and an open invariant subset $U\subseteq \mg^{(0)}$, we establish a graded isomorphism  
\begin{equation}\label{tttt11}
A_R(\mg_U)\cong_{\gr}C_R(U)\rtimes\mg^{(h)}.
\end{equation}
Here $\mg^{(h)}$ is the inverse semigroup of graded compact open bisections of $\mg$ which acts partially on $U$, $C_R(U)\rtimes\mg^{(h)}$ is the  corresponding partial skew inverse semigroup ring and $A_R(\mg_U)$ is the Steinberg algebra associated to the groupoid $\mg_U=r^{-1}(U)$.  In particular, we have a graded isomorphism $A_R(\mg)\cong_{\gr}C_R(\mg^{(0)})\rtimes_{\pi}\mg^{(h)}$ (see Theorem \ref{thmpsisr}).

%Let $X$ be a locally compact Hausdorff topological space which is totally disconnected, $G$ a discrete group and $R$ a commutative ring with identity. We denote by $C_R(X)$ the set of $R$-valued continuous function (i.e., locally constant) with compact support. $C_R(X)$ is an $R$-algebra with pointwise multiplication. 

%Let $\mg$ be a graded ample Hausdorff groupoid, that is $\mg$ is an ample Hausdorff groupoid with a continuous $1$-cocycle $c:\mg\xra G$. We have an inverse semigroup $\mg^{(h)}$ which is the collection of graded compact open bisections of $\mg$. 

%In this note, we realise the graded Steinberg algebra of $\mg$ as the partial skew inverse semigroup ring $C_R(\mg^{0})\rtimes_{\pi}\mg^{(h)}$  

Let $\phi=(\phi_g, X_g, X)_{g\in G}$ be a partial action of a group $G$ on $X$ such that each $X_g$ is clopen subset of $X$. Then we have an induced partial action of $G$ on $C_R(X)$. Denote by $\mg_X=\bigcup_{g\in G}g\times X_g$ the $G$-graded groupoid given in \eqref{groupoid} associated to $\phi$. As a direct consequence of Theorem \ref{thmpsisr}, we realise the $G$-graded Steinberg algebra $A_R(\mg_X)$ as the partial skew group ring $C_R(X)\rtimes_{\phi}G$ (Proposition \ref{propss}). 
Specialising to the setting of the Leavitt path algebra of a directed graph $E$, we recover Gon\c{c}alves and Royer's result \cite[Theorem 3.3]{goncalvesroyer}, showing that $L_R(E)$ is graded isomorphic to the partial skew group ring via an isomorphism of groupoids (Corollary \ref{corlpa}). 
For the case $G=\Z$, we give a condition on the directed graph $E$ so that the Leavitt path algebra of $E$ can be realised as the partial skew group ring of $\Z$ on $C_R(X)$.

Recall that when $\mg$ is an ample, Hausdorff groupoid with a continuous cocycle $c : \mg\xra G$ and $c^{-1}(\varepsilon)$ is strongly effective, there is a one-to-one correspondence between the open invariant subsets of $\mg^{(0)}$ and the graded ideals of the Steinberg algebra $A_K (\mg)$ of $\mg$ (see \cite[Theorem 5.3]{cep}). We observe that the associated groupoid $\mg_X$ is an ample, Hausdorff groupoid if $X$ is locally compact Hausdorff with a basis of compact open sets. Applying Proposition \ref{propss} there is a one-to-one correspondence between the open invariant subsets of $\mg_X^{(0)}$ and the graded ideals of the partial skew group ring $C_R(X)\rtimes_{\phi}G$.

Exel constructed the partial group $C^*$-algebra \cite[Definition 6.4]{exel1998} of a group $G$.  We consider the algebraic version of the construction and realise it as a partial skew group ring associated to a partial action of an abelian group $G$ on an appropriate set.

The paper is organised as follows. In Section \ref{section2}, we realise the Steinberg algebra of $\mg$ as a partial skew inverse semigroup ring. In Section \ref{section3}, we consider the partial action of $G$ on $X$. In subsection \ref{subsection31}, we prove that the $G$-graded Steinberg algebra $A_R(\mg_X)$ is graded isomorphic to the partial skew group ring $C_R(X)\rtimes_{\phi}G$. In subsection \ref{subsection32}, we prove the one-to-one correspondence between the open invariant subsets of $\mg_X^{(0)}$ and the graded ideals of the partial skew group ring. In subsection \ref{subsection33}, we give the example of a partial skew group ring arising from an abelian group $G$. In Section \ref{section4}, we realises Leavitt path algebras as graded partial skew group ring, where the grading is over a free group. We also describe a graph condition under which the associated Leavitt path algebra can be realised as a $\Z$-graded partial skew group ring.

\emph{Historical notes.} While preparing this paper, Beuter and Gon\c{c}alves posted \cite{beutergoncalves} on arXiv which contains two main theorems: Theorem 3.2 in \cite{beutergoncalves} proves that $C_R(X)\rtimes_{\phi} G$ is isomorphic to $A_R(\mg_X)$ as $R$-algebras and Theorem 5.2 realises Steinberg algebra of an ample Hausdorff groupoid as a partial skew inverse semigroup ring. Our Theorem \ref{thmpsisr} improve their Theorem 5.2 by showing that there is a graded isomorphism on the level of ideals and that Theorem 3.2 in \cite{beutergoncalves}  is direct consequence of Theorem \ref{thmpsisr}.

\section{Steinberg algebras and partial skew inverse semigroup rings}\label{section2}

In this section, we consider a $G$-graded ample Hausdorff groupoid $\mg$, where $G$ is a discrete group and its associated Steinberg algebra $A_R(\mg)$.  The main result of this section (Theorem~\ref{thmpsisr}) is to realise the $G$-graded Steinberg algebra  $A_R(\mg)$ as a partial skew inverse semigroup ring. We briefly recall the concepts of Steinberg algebras and partial skew semigroup rings. 

\subsection{Steinberg algebras} Steinberg algebras were introduced in~\cite{st} in the context of discrete inverse
semigroup algebras and independently in \cite{cfst} as a model for Leavitt path algebras.

A groupoid $\mg $ is a small category in which every morphism is invertible. It can also be
viewed as a generalisation of a group which has a partial binary operation (and several units).   If $x\in\mg$, $d(x)=x^{-1}x$ is the \emph{domain} of $x$ and $r(x)=xx^{-1}$ is
its \emph{range}. The pair $(x,y)$ is composable if and only if $r(y)=d(x)$. The set
$\mg^{(0)}:=d(\mg)=r(\mg)$ is called the \emph{unit space} of $\mg$. Elements of
$\mg^{(0)}$ are units in the sense that $xd(x)=x$ and $r(x)x=x$ for all $x \in \mg$. A subset $U$ of the
unit space $\mg^{(0)}$ of a groupoid $\mg$ is called \emph{invariant} if $d(\g)\in U$ implies $r(\g)\in U$;
equivalently,
\[
    r(d^{-1}(U))=U=d(r^{-1}(U)).
\] For
$U,V\sub\mg$, we define
\[
    U^{-1}=\{\g^{-1}\;|\; \g\in U\};~~~~~~~ ~~~~~~~UV=\big \{\a\b \mid \a\in U,\b\in V \text{ and } r(\b)=d(\a)\big\}.
\]

A topological groupoid is a groupoid endowed with a topology under which the inverse map
is continuous, and such that composition is continuous with respect to the relative
topology on $\mg^{(2)} := \{(\b,\g) \in \mg \times \mg : d(\b) = r(\g)\}$ inherited from
$\mg\times \mg$. An \emph{\'etale} groupoid is a topological groupoid $\mg$ such that the
domain map $d$ is a local homeomorphism. In this case, the range map $r$ is also a local
homeomorphism. An \emph{open bisection} of $\mg$ is an open subset $U\subseteq \mg$ such
that $d|_{U}$ and $r|_{U}$ are homeomorphisms onto an open subset of $\mg^{(0)}$. We say
that an \'etale groupoid $\mg$ is \emph{ample} if there is a basis consisting of compact
open bisections for its topology.

Let $\mg$ be an ample Hausdorff topological groupoid. Suppose that $R$ is a commutative ring with identity. Consider $A_R(\mg) = C_c(\mg, R)$, the space of
compactly supported continuous functions from $\mg$ to $R$ with $R$ given the discrete topology.
Then $A_R(\mg)$ is an $R$-algebra with addition defined point-wise and
multiplication is given by convolution
$$(f*g)(\g) = \sum_{\{\a\b=\g\}}f(\a)g(\b).$$
It is useful to note that $$1_{U}*1_{V}=1_{UV}$$ for compact open bisections $U$ and $V$
(see \cite[Proposition 4.5(3)]{st}). With this structure, $A_R(\mg)$ is an algebra called the \emph{Steinberg algebra}  associated to $\mg$. The algebra $A_R(\mg)$ can also be realised as the span of characteristic functions of the form $1_U$ where $U$ is a compact open bisection (see \cite[Lemma 3.3]{cfst}).

Let $G$ be a group with identity $\varepsilon$. A ring $A$ (possibly without unit)
is called a \emph{$G$-graded ring} if $ A=\bigoplus_{g \in G} A_{g}$
such that each $A_{g}$ is an additive subgroup of $A$ and $A_{g}  A_{h}
\subseteq A_{gh}$ for all $g, h \in G$. The group $A_g$ is
called the $g$-\emph{homogeneous component} of $A.$ When it is clear from context
that a ring $A$ is graded by a group $G,$ we simply say that $A$ is a  \emph{graded
ring}. If $A$ is an algebra over a ring $R$, then $A$ is called a \emph{graded algebra}
if $A$ is a graded ring and $A_{g}$ is a $R$-submodule for any $g \in G$.

The elements of $\bigcup_{g \in G} A_{g}$ in a graded ring $A$ are called
\emph{homogeneous elements} of $A.$ The nonzero elements of $A_g$ are called
\emph{homogeneous of degree $g$} and we write $\deg(a) = g$ for $a \in
A_{g}\backslash \{0\}.$ We say that a $G$-graded ring $A$ is
\emph{trivially graded} if $A_{\varepsilon}$ is the only nonzero component of $A$, that is, $A_\varepsilon=A$, so $A_g=0$ for $g \in G
\backslash \{\varepsilon\}$. Any ring admits a trivial grading by any group.

Let $G$ be a discrete group and $\mg$ a topological groupoid. A $G$-grading of $\mg$ is
a continuous function $c : \mg \to G$ such that $c(\a)c(\b) = c(\a\b)$ for all $(\a,\b)
\in \mg^{(2)}$; such a function $c$ is called a continuous $1$-cocycle on $\mg$. In this case, we call $\mg$ a $G$-graded ample groupoid and we write $\mg=\bigsqcup_{g\in G}\mg_g$, where $\mg_g=c^{-1}(g)$. Note that $\mg_g \mg_h \subseteq \mg_{gh}$. 
We say a subset $X\subseteq \mg$ is $g$-graded if $X\subseteq \mg_g$ for $g\in G$. One can see that $\mg^{(0)} \subseteq \mg_\varepsilon$, i.e.,  the unit space  is $\varepsilon$-graded, where $\varepsilon$ is the identity of the group $G$.

Recall from \cite[Lemma 3.1]{cs} that if $\mg=\bigsqcup_{g\in G}\mg_g$ is a $G$-graded groupoid, then the Steinberg algebra $A_R(\mg)$ is a $G$-graded
algebra with homogeneous components
\begin{equation}\label{hgboat}
    A_{R}(\mg)_{g} = \{f\in A_{R}(\mg)\mid \supp(f)\subseteq \mg_g\}.
\end{equation}
The family of all idempotent elements of $A_{R}(\mg^{(0)})$ is a set of local units for
$A_{R}(\mg)$ (\cite[Lemma 2.6]{cep}). Here, $A_{R}(\mg^{(0)})\subseteq A_{R}(\mg)$ is a
subalgebra. Note that any ample Hausdorff
groupoid admits the trivial cocycle from $\mg$ to the trivial group $\{\varepsilon\}$,
which gives rise to a trivial grading on $A_{R}(\mg)$.

\subsection{Partial skew inverse semigroup rings}\label{subsection22}
An inverse semigroup is a semigroup $\SS$ such that for each $s\in\SS$, there exists a unique $s^*\in\SS$ such that $ss^*s=s$ and $s^*ss^*=s^*$. We call $s^*$ the inverse of $s$. Note that every group is an inverse semigroup. Recall that there is a natural partial order relation defined on an inverse semigroup $\SS$ by $$s \leq t \Longleftrightarrow s=ts^*s\Longleftrightarrow s=ss^*t$$ for $s,t\in\SS$.

\begin{deff}\label{defpar} \cite[Proposition 3.5]{bussexel} Let $\SS$ be an inverse semigroup. A \emph{partial action} of $\mathcal{S}$ on a set $X$ is $\pi=(\pi_s, X_s, X)_{s\in \SS}$ with $X_s\sub X$ a subset and $\pi_s: X_{s^*}\xra X_s$ a bijection such that for all $s, t \in\SS$

\begin{itemize}
\item[(i)] $\pi^{-1}_s=\pi_{s^*}$;

\item[(ii)] $\pi_s(X_{s^*}\cap X_t)\sub X_{st}$;

\item[(iii)] if $s\leq t$, then $X_s\sub X_t$;

\item[(iv)] For every $x\in X_{t^*}\cap X_{t^*s^*}$, $\pi_s(\pi_t(x))=\pi_{st}(x)$.
\end{itemize}
\end{deff}

If $\SS$ has a zero element $0$, we assume that $X_0=\emptyset$. In case that $X$ is an algebra or a ring then the subsets $X_s$  should also be ideals and the maps $\pi_s$ should be isomorphisms of algebras. In the topological setting, each $X_s$ should be an open set and each $\pi_s$ a homeomorphism of topological spaces. Furthermore, if the inverse semigroup $\SS$ has a unit $\varepsilon$, we assume that $X_{\varepsilon} = X$ and $\pi_{\varepsilon}$ is the identity map of $X$.

Let $\pi=(\pi_s, A_s, A)_{s\in\SS}$ be an action of the inverse semigroup $\SS$ on an algebra $A$. Define $\LL$ as the set of all formal forms $\sum_{s\in \SS}a_s\d_s$ (with finitely many $a_s$ nonzero), where $a_s\in A_s$ and $\d_s$ are symbols, with addition defined in the obvious way and multiplication being the linear extension of $$(a_s\d_s) (a_t\d_t)=\pi_s\big(\pi_{s^{-1}}(a_s)a_t\big)\d_{st}.$$ Then $\LL$ is an algebra which is possibly not associative. Exel and Vieira proved under which condition $\LL$ is associative (see \cite[Theorem 3.4]{exelvieira}). Combing \cite[Theorem 3.4] {exelvieira} with \cite[Proposition 2.5]{de}, if each ideal $A_s$ is idempotent or non-degenerate, then $\LL$ is associative.

\begin{deff} \label{defpartialinverse}Let $\pi=(\pi_s, A_s, A)_{s\in\SS}$ be an action of an  inverse semigroup $\SS$ on an algebra $A$. Consider $\mathcal{N}=\langle a\d_s-a\d_t: a\in A_s, s\leq t\rangle$, which is the ideal generated by $a\d_s-a\d_t$. The \emph{partial skew inverse semigroup ring} $A\rtimes_{\pi}\SS$ is defined as the quotient ring $\LL/ \mathcal{N}$. 
\end{deff}

Next we equip these algebras with a graded structure. Suppose that $G$ is a group and $w: \SS\setminus\{0\}\xra G$ is a function such that \begin{equation}
\label{function}
w(st)=w(s)w(t)
\end{equation} for $s, t\in\SS$ with $st\neq 0$. Here, $0$ is  a zero element of $\SS$. If $\SS$ does not have a zero element then the function $w$ is a homomorphism from $\SS$ to $G$. For each non-zero element $s\in\SS$ since $s=ss^*s$ we have $w(s^*)=w(s)^{-1}$. 

Observe now that the algebra $\LL$ is a $G$-graded algebra with elements $a_s\d_s\in\LL$ with $a_s\in A_s$ are homogeneous elements of degree $w(s)$. Furthermore, if $s\leq t$, then $s=ts^*s$. It follows that $w(s)=w(t)w(s^*)w(s)=w(t)$. Hence $a\d_s-a\d_t$ with $s\leq t$ and $a\in A_s$ is a homogeneous element in $\LL$. Thus the ideal $\mathcal N$ generated by homogeneous elements is a graded ideal and therefore the quotient algebra $A\rtimes_{\pi}\SS= \LL/\mathcal N$ is $G$-graded.

Let $X$ be a Hausdorff topological space and $R$ a unital commutative ring with a discrete topology. Let $C_R(X)$ be the set of $R$-valued continuous function (i.e., locally constant) with compact support.  If $D$ is a compact open subset of $X$, the characteristic function of $D$, denoted by $1_D$, is clearly an element of $C_R(X)$. In fact, every $f$ in $C_R(X)$ may be written as 
 \begin{equation}\label{rep}
 f=\sum_{i=1}^nr_i1_{D_i}, 
 \end{equation} where $r_i\in R$ and the $D_i$ are compact open, pairwise disjoint subsets of $X$. $C_R(X)$ is a commutative $R$-algebra with pointwise multiplication. The support of $f$, defined
by $\supp(f)=\{x\in X\;|\; f(x)\neq 0\}$, is clearly a compact open subset.

We observe that $C_R(X)$ is an idempotent ring. We have \begin{equation}\label{equationforlocal}
 \sum_{i=1}^n1_{D_i} \cdot f=f \cdot \sum_{i=1}^n1_{D_i}=\sum_{i=1}^nr_i1_{D_i}=f
 \end{equation} for any $f\in C_R(X)$ which is written as \eqref{rep}. So $C_R(X)$ is a ring with local units and thus an idempotent ring. 

Let $\mg$ be a $G$-graded ample Hausdorff groupoid.  Set \begin{equation}
\mg^{(h)}=\{B\; |\; B \text{~is a graded compact open bisection of~} \mg\}.
\end{equation} Then $\mg^{(h)}$ is an inverse semigroup such that $B^*=B^{-1}$ for each $B\in\mg^{(h)}$ (see \cite[Proposition 2.2.4]{paterson}).  Furthermore, the map $\mg^{(h)} \backslash \emptyset \rightarrow G, U\mapsto g$, if $U\subseteq \mg_g$, makes $\mg^{(h)}$ a graded inverse semigroup. 
 Observe that in  the inverse semigroup $\mg^{(h)}$, $B\leq C$ if and only if $B\sub C$ for $B, C\in \mg^{(h)}$. 

Let $U\sub \mg^{(0)}$ be an open invariant subset. For each $B\in \mg^{(h)}$, define $$U_B=r(B)\cap U, \;\;\; U_{B^{-1}}=d(B)\cap U$$ and $\pi_B: U_{B^{-1}}\xra U_B$ is given by $
\pi_B(u)=r(b)$ if $u=d(b)$ for some $b\in B$. The map $\pi_B$ is well defined, since $B$ is a bisection of $\mg$. Observe that $\pi_B$ is a bijection with inverse $\pi_{B^{-1}}$. We claim that $\pi_B$ is a homeomorphism for each $B\in \mg^{(h)}$. Take any open subset $O\sub U_B$. Observe that $\pi_B^{-1}(O)=d(r^{-1}(O)\cap B)$ is an open subset of $U_{B^{-1}}$. Thus $\pi_B$ is continuous. Similarly, $\pi_B^{-1}$ is continuous.

Now we show that $\pi=(\pi_B, U_B, U)_{B\in\mg^{(h)}}$ is a partial action of $\mg^{(h)}$ on $U$. We check conditions (i)--(iv) in Definition \ref{defpar}. For (i), we have $\pi_{B^{-1}}=\pi_B^{-1}$ by the definition of $\pi_B$ and $\pi_{B^{-1}}$. For (ii), observe that $$\pi_B(U_{B^{-1}}\cap U_C)=\pi_B(U\cap d(B)\cap r(C))\sub d(BC)\cap U=U_{BC}$$ for any $B, C\in \mg^{(h)}$. For (iii), $U_B\sub U_C$ if $B\sub C$. For (iv), observe that for $u\in U_{C^{-1}}\cap U_{C^{-1}B^{-1}}$, we have $u=d(c)=d(bc)$ for some $c\in C$, $b\in B$ with $r(c)=d(b)$. It follows that $$\pi_B(\pi_C(u))=\pi_B(r(c))=\pi_B(d(b))=r(b)=r(bc)=\pi_{BC}(u),$$ and thus (iv) holds.

There is an induced partial action $(\pi_B, C_R(U_B), C_R(U))_{B\in \mg^{(h)}}$ of $\mg^{(h)}$ on an algebra $C_R(U)$. Here the map $\pi_B: C_R(U_{B^{-1}})\xra C_R(U_B)$ is given by $\pi_B(f)=f\circ \pi_B^{-1}$. We still denote the induced partial action by $\pi$. In this case, $\LL=\{\sum_{B\in \mg^{(h)}}a_B\d_B\;|\; a_B\in C_R(U_B)\}$ is associative, since each ideal $C_R(U_B)$ is idempotent. Since $\mg^{(h)}$ is $G$-graded, $C_R(U)\rtimes_{\pi} \mg^{(h)}$ is a $G$-graded algebra. 

Let $B$ be an $R$-algebra. A \emph{representation} \cite{clarkmichell,cfst} of $\mg^{(h)}$ in $B$ is a family
$\{t_D :D\in \mg^{(h)}\}\sub B$ satisfying
\begin{itemize}
\item[(R1)] $t_{\emptyset}= 0$;

\item[(R2)] $t_Dt_C=t_{DC}$  for all $D,C\in \mg^{(h)}$; and

\item[(R3)] $t_D + t_C=t_{D\cup C}$, whenever $D$ and $C$ are disjoint elements of $\mg^{(h)}\cap \mg_\g$ for some $\g\in  G$ such that $D\cup C$ is a bisection.
\end{itemize}

For each open invariant subset $U\sub \mg^{(0)}$, we denote $d^{-1}(U)$ by $\mg_U$ which is an open subgroupoid of $\mg$. The following is the main result of this section.

\begin{thm} \label{thmpsisr} Let $\mg$ be a $G$-graded ample Hausdorff groupoid, $U\sub \mg^{(0)}$ an open invariant subset and $\pi=(\pi_B, C_R(U_B), C_R(U))_{B\in\mg^{(h)}}$ the induced partial action of $\mg^{(h)}$ on $C_R(U)$. Then there is a $G$-graded isomorphism of $R$-algebras
\begin{equation}
A_R(\mg_U)\cong_{\gr}C_R(U)\rtimes_{\pi}\mg^{(h)}.
\end{equation} In particular, we have $A_R(\mg)\cong_{\gr}C_R(\mg^{(0)})\rtimes_{\pi}\mg^{(h)}$. 
\end{thm}

\begin{proof} We define a representation $\{t_D\;|\; D \in\mg_U^{(h)}\}$ in the algebra $C_R(U)\rtimes_{\pi}\mg^{(h)}$. For each $D\in \mg_U^{(h)}$, define $t_D=1_{r(D)}\d_D$. For (R1), $t_{\emptyset}=0$. For (R2), we have \begin{equation*}
\begin{split}
t_Dt_C&=1_{r(D)}\d_D1_{r(C)}\d_C=\pi_D\big(\pi_D^{-1}(1_{r(D)})1_{r(C)}\big)\d_{DC}=\pi_D(1_{d(D)}1_{r(C)})\d_{DC}\\
&=\pi_D(1_{d(D)\cap r(C)})\d_{DC}=1_{r(DC)}\d_{DC}=t_{DC},
\end{split}
\end{equation*} for any $D, C\in\mg^{(h)}_U$. For (R3), suppose that $D\sub \mg_{\g}$ and $C\sub \mg_{\g}$ are disjoint elements of $\mg^{(h)}_U$ for some $\g\in G$ such that $D\cup C$ is a bisection of $\mg$. It follows that \begin{equation*}
\begin{split}
t_D+t_C&=1_{r(D)}\d_D+1_{r(C)}\d_C
=1_{r(D)}\d_{D\cup C}+1_{r(C)}\d_{D\cup C}
=(1_{r(D)}+1_{r(C)})\d_{D\cup C}
=t_{D\cup C}, 
\end{split}
\end{equation*} where the equations hold because $D\cup C$ is a bisection and thus $r(D\cup C)=r(D)\cup r(C)$. 

By the universality of Steinberg algebras (refer to \cite[Theorem 3.10]{cfst} and \cite[Proposition 2.3]{clarkmichell}),  we have an $R$-homomorphism $$f: A_R(\mg_U)\longrightarrow C_R(U)\rtimes_{\pi}\mg^{(h)}$$ such that $f(1_D)=t_D$ for each $D\in \mg^{(h)}_U$. It is evident that $f$ preserves the grading. Hence, $f$ is a homomorphism of $G$-graded algebras.

To prove the surjectivity, fix any $B\in \mg^{(h)}$. Suppose that $D\sub U_B$ is a compact open subset. Recall that $U_B=r(B)\cap U$. We first show that $1_D\d_B\in {\rm Im} f$. Set $E=r^{-1}(D)\cap B$. Observe that $E$ is a graded compact open bisection of $\mg_U$. We claim that $D=r(E)$. Obviously, $r(E)\sub D$. Conversely, for any $v\in D\sub r(B)\cap U$ there exists $b\in B$ such that $v=r(b)\in U$. We have $b\in r^{-1}(D)\cap B=E$. Thus $v=r(b)\in r(E)$. The proof of claim is completed. Hence, we have \begin{equation}
\label{surj}
1_D\d_B=1_{r(E)}\d_B=1_{r(E)}\d_E=f(1_E)\in {\rm Im} f,
\end{equation} where we use the fact that $E\leq B$ in $\mg^{(h)}$. Next for any $a_B\d_B\in C_R(U)\rtimes_{\pi}\mg^{(h)}$ with $B\in \mg^{(h)}$, we can write $a_B$ as $\sum_{i=1}^nr_i1_{D_i}$ with $D_i$'s mutually disjoint compact open subset of $U_B$ and $r_i\in R$. By \eqref{surj}, we have $a_B\d_B\in {\rm Im}f$.

To prove injectivity of $f$, we define a map $g: C_R(U)\rtimes_{\pi}\mg^{(h)}\xra A_R(\mg_U)$. For each $B\in \mg^{(h)}$ and $a_B\d_B\in \mathcal{L}$, we define \begin{equation*}
g(a_B\d_B)=
\begin{cases}
a_B(r(x)), & \text{~if~} x\in B,\\
0, & \text{otherwise},
\end{cases}
\end{equation*} and extend it linearly to $\mathcal{L}$. Observe that the support of $g(a_B\d_B)$ is $B\cap r^{-1}\big(\supp (a_B)\big)$, for each $x\in \mg_U\cap B$, we have $x\in r^{-1}(N')\cap \mg_U\sub \big(r^{-1}(N)\cap \mg_U\big)$ with $r(x)\in N'\sub N$ and $N'\sub U_B$ an open subset, and for $x\in \mg_U$ and $x\notin B$ the set $\mg_U\setminus B$ is a neighbourhood of $x$ such that the restriction of $g(a_B\d_B)$ to it is zero. It follows that $g(a_B\d_B)$ is locally constant with compact support. For $B\leq C$ in $\mg^{(h)}$ and $a\in C_R(U_B)$, we have $g(a\d_B)=g(a\d_C)$ because of $r(x)\in r(C)\setminus r(B)$ and $r(x)\notin U_B$ for $x\in C\setminus B$.

Now we check that $gf=\id_{A_R(\mg_U)}$, which implies that $f$ is injective. It is evident that $(gf)(1_B)=g(1_{r(B)}\d_B)=1_B$ for each $B\in\mg_U^{(h)}$. For each element $a\in A_R(\mg_U)$, we write $a=\sum_{i=1}^nr_i1_{B_i}$ with $r_i\in R\setminus\{0\}$ and the $B_i$'s mutually disjoint graded compact open bisections of $\mg_U$. Obviously we have $(gf)(a)=a$, implying $gf=\id_{A_R(\mg_U)}$.
\end{proof}

\section{Partial skew group rings}\label{section3}

In this section, we consider a partial action of a discrete group on a locally compact Hausdorff topological space which has a basis of compact open sets. Abadie \cite[\S3]{abadie} proved that the $C^*$-algebra of the associated groupoid of the partial action agrees with the crossed product by the partial action. Given a partial action of a discrete group on a set, the partial skew group ring is defined as an algebra associated to it. Applying Theorem \ref{thmpsisr}, we prove that the Steinberg algebra of the associated groupoid of the partial action is graded isomorphic to the partial skew group ring. We describe the graded ideals of the partial skew group ring and give a partial skew group ring which arises from a group.

\subsection{Partial skew group rings} \label{subsection31}

A partial action of a group, appeared in various areas of mathematics, in particular, in the theory of operator algebras as a powerful tool in their study (see \cite{abadie1,exel1997,quiggraeburn}).

We recall the definition of partial action of a group on a set. Let $G$ be a group with $\varepsilon$ the identity of the group. A \emph{partial action} \cite{ferrero} of $G$ on a set $X$ is a data $\phi=(\phi_{g}, X_g, X)_{g\in G}$, where for each $g\in G$, $X_g$ is a subset of $X$ and $\phi_g:X_{g^{-1}}\xra X_g$ is a bijection such that 
\begin{itemize}
\item[(i)]$X_{\varepsilon} =X$ and $\phi_{\varepsilon}$ is the identity on $X$, where $\varepsilon$ is the identity of the group $G$;

\item[(ii)] $\phi_g(X_{g^{-1}}\cap X_h) = X_g\cap X_{gh}$ for all $g, h\in G$;

\item[(iii)] $\phi_g(\phi_h(x)) =\phi_{gh}(x)$ for all $g, h\in G$ and $x\in X_{h^{-1}}\cap X_{h^{-1}g^{-1}}$.
\end{itemize}
Since a group is an inverse semigroup, the definition for a partial action of $G$ on a set $X$ can be obtained by Definition \ref{defpar} and \cite[Proposition 3.4]{bussexel}. In case that $X$ is a topological space, each $X_g\sub X$ is an open subset and each $\phi_g:X_{g^{-1}}\xra X_g$ is a homoemorphism.

\begin{lemma} \label{inducedpartial}
Let $\phi=(\phi_g, X_g, X)_{g\in G}$ be a partial action of a group $G$ on a set $X$ and $\psi:G\rightarrow H$ a group homomorphism. Then $\psi$ induces a partial action $\overline{\phi}$ of $H$ on $X$ if and only if for every $g \in \ker\psi$, $X_g=X_{g^{-1}}$ and $\phi_g=\id$. 
\end{lemma}
\begin{proof}
We first observe that if $g \in \ker\psi$, then $X_g=X_{g^{-1}}$ and $\phi_g=\id$ is equivalent to if $\psi(g_1)=\psi(g_2)$, then $\phi_{g_1}|_{X_{g_1^{-1}}\cap X_{g_2^{-1}}}=\phi_{g_2}|_{X_{g_1^{-1}}\cap X_{g_2^{-1}}}$. Let $x\in X_{g_1^{-1}}\cap X_{g_2^{-1}}$, where $\psi(g_1)=\psi(g_2)$. Then $\phi_{g_1}(x) \in X_{g_1}\cap X_{g_1g_2^{-1}}$. Since $g_2g_1^{-1} \in \ker \psi$, $X_{g_1g_2^{-1}}=X_{g_2g_1^{-1}}$ and $\phi_{g_2 g_1^{-1}}=\id$. Thus 
\[
\phi_{g_1}(x) =\phi_{g_2 g_1^{-1}}(\phi_{g_1}(x))=\phi_{g_2}\phi_{g_1^{-1}}(\phi_{g_1}(x))=\phi_{g_2}(x).
\] On the other hand, suppose $\psi(g)={\varepsilon}$. Then we have $\phi_g|_{X_{g^{-1}}\cap X_{\varepsilon}}=\phi_{\varepsilon}|_{X_{g^{-1}}\cap X_{\varepsilon}}$. Since $X_{\varepsilon}=X$ and $\phi_{\varepsilon}=\id$, we immediately obtain that $X_g=X_{g^{-1}}$ and $\phi_g=\id$. 

Suppose now we have an induced partial action $H$ on $X$. Then for $h\in H$, $X_h=\bigcup_{\psi(g)=h}X_g$ and $\overline{\phi}_{h}:X_{h^{-1}}\rightarrow X_h$ defined by $\overline{\phi}_{h}(x):=\phi_g(x)$, where $\psi(g)=h$, is well-defined. Thus if $\psi(g_1)=\psi(g_2)$, then we should have $\phi_{g_1}|_{X_{g_1^{-1}}\cap X_{g_2^{-1}}}=\phi_{g_2}|_{X_{g_1^{-1}}\cap X_{g_2^{-1}}}$. By the first part of the proof it follows that for every $g \in \ker\psi$, $X_g=X_{g^{-1}}$ and $\phi_g=\id$. 

Conversely, considering $X_h=\bigcup_{\psi(g)=h}X_g$ and defining $\overline{\phi}_{h}:X_{h^{-1}}\rightarrow X_h$ by \begin{equation}\label{definduced}
\overline{\phi}_{h}(x):=\phi_g(x),
\end{equation} where $\psi(g)=h$, by the assumption it follows that  $\overline{\phi}_{h}$ is well-defined. We first show that $\overline{\phi}_{h}$ is bijective. Suppose $\overline{\phi}_{h}(x_1)=\overline{\phi}_{h}(x_2)$. There are $g_1, g_2\in G$ such that $\psi(g_1)=\psi(g_2)=h$, $x_1\in X_{g_1^{-1}}$, $x_2\in X_{g_2^{-1}}$ and $\phi_{g_1}(x_1)=\phi_{g_2}(x_2)\in X_{g_1}\cap X_{g_2}$. Thus 
\[x_2=\phi_{g_2^{-1}}\phi_{g_2}(x_2)=\phi_{g_2^{-1}}\phi_{g_1}(x_1)=\phi_{g_2^{-1}g_1}(x_1)=\id(x_1)=x_1.
\] Surjectivity of $\overline{\phi}_{h}$ and the rest of the conditions of partial actions are straightforward. 
\end{proof}

\begin{cor}
Let $\phi=(\phi_g, X_g, X)_{g\in G}$ be a partial action of a group $G$ on a set $X$ and $H$ a normal subgroup of $G$. Then there is an induced partial action 
$\overline{\phi}$ of $G/H$ on $X$ if and only if for every $h \in H$, $X_h=X_{h^{-1}}$ and $\phi_h=\id$. 
\end{cor}

Suppose that $X$ is a topological space, $\phi=(\phi_g, X_g, X)_{g\in G}$ a partial action of a group $G$ on $X$ and $\psi:G\xra H$ a group homomorphism. Then $\psi$ induces a partial action $\overline{\phi}$ of $H$ on $X$ if and only if for every $g \in \ker\psi$, $X_g=X_{g^{-1}}$ and $\phi_g=\id$. We only need to check that the map $\overline{\phi}_h$ given in \eqref{definduced} is a homeomorphism for each $h\in G$. First, we show that $\overline{\phi}_h$ is continuous. For any open subset $O$ of $X_h=\bigcup_{\psi(g)=h}X_g$, we have $\overline{\phi}_h^{-1}(O)=\bigcup_{\psi(g)=h}\phi_{g^{-1}}(O\cap X_{g^{-1}})$. It follows that $\overline{\phi}_h^{-1}(O)$ is open, and thus $\overline{\phi}_h$ is continuous. Similarly, we can prove that $\overline{\phi}^{-1}_h$ is continuous. So $\overline{\phi}_h$ is a homeomorphism for each $h\in G$.

The group $G$ is an inverse semigroup with $g\leq h$ when $g=h$ in $G$. The partial skew inverse semigroup ring $A \rtimes_{\phi} G$ equals the algebra $\LL$ given in Definition \ref{defpartialinverse}. The ring $A \rtimes_{\phi} G$ is called the \emph{partial skew group ring} \cite[Definition 1.2]{de}  corresponding to $\phi$. Observe that $A \rtimes_{\phi}G$ is always a $G$-graded ring with $(A \rtimes_{\phi}G)_g=A_g  \d_g$ induced by the identity map from $G$ to $G$.

Our first result shows that similar to skew group rings, one can lift some of the properties on $A_g$, $g\in G$ to the whole ring $A \rtimes_{\phi} G$. Recall that a ring $A$ (not necessarily unital) is called \emph{von Neumann regular} in case for every $x \in A$ there exists $y\in A$ such that $x = xyx$. A $G$-graded ring $A$ is called a \emph{graded von Neumann regular} ring, if for any homogeneous element $x \in A$, there is $y\in A$ such that $xyx = x$. Note that $y$ can be chosen to be a homogeneous element. Throughout the note, we call such rings also graded regular rings.

\begin{lemma}\label{lemmavon}
Let $\phi=(\phi_g, X_g, X)_{g\in G}$ be a partial action of $G$ on a  ring $A$ and $A\rtimes_\phi G$ the partial skew group ring. Then  $A\rtimes_\phi G$ is graded von Neumann regular if and only if $A_g$, $g \in G$, are von Neumann regular. 
\end{lemma}

\begin{proof} Suppose that $A  \rtimes_\phi \G$ is graded regular. For any homogeneous element $a_g\delta_g\in A  \rtimes_\phi \G$, there exists an element $b_{g^{-1}}\delta_{g^{-1}}$ such that $a_g\delta_g \cdot b_{g^{-1}}\delta_{g^{-1}} \cdot a_g\delta_g=a_g\delta_g$. We observe that $a_g\delta_g \cdot b_{g^{-1}}\delta_{g^{-1}} \cdot a_g\delta_g=\phi_g(\phi_{g^{-1}}(a_g) b_{g^{-1}})\delta_{\varepsilon}\cdot a_g\delta_g=\phi_g(\phi_{g^{-1}}(a_g))\phi_g(b_{g^{-1}})a_g\delta_g=a_g \phi_g(b_{g^{-1}})a_g\delta_g=a_g\delta_g$.  Thus $a_g=a_g \phi_g(b_{g^{-1}})a_g$ and $A_g$ is regular for each $g\in G$. Conversely, the statement follows similarly. 
\end{proof}

 From now on, we assume that $X$ is a locally compact Hausdorff topological space which is totally disconnected and $G$ is a discrete group. Recall that a topological space is said to be totally disconnected if its topology is generated by clopen subsets. Let $\phi=(\phi_g, X_g, X)_{g\in G}$ be a partial action of $G$ on $X$. We assume that each $X_g$ is a clopen subset of $X$.  Consider the $G$-graded groupoid \begin{equation}
 \label{groupoid}
 \mg_X=\bigcup_{g \in G} g \times X_g,
 \end{equation} whose composition and inverse maps are given by $(g,x)(h,y)=(gh,x)$ and $(g,x)^{-1}=(g^{-1},\phi_{g^{-1}}(x))$. Here the range and source maps are given by $r(g,x)=(\varepsilon, x)$, $d(g,x)=(\varepsilon, \phi_{g^{-1}}(x))$ with $\varepsilon$ the identity of $G$.  The unit space of $\mg_X$ is identified with $X$. The topology of $\mg_X$ is inherited from the product topology $G\times X$. Notice that $\mg_X$ is a locally compact Hausdorff groupoid. 

 The groupoid $\mg_X$ is ample. In fact, $\mg_X$ admits a left Haar system (refer to \cite[Proposition 2.2]{abadie}). Combing \cite[Proposition 2.8]{re}, the range map $r:\mg_X\xra \mg_X^{(0)}$ is a local homeomorphism, equivalently $\mg_X$ is $\acute{\rm e}$tale. Recall that a locally compact Hausdorff $\acute{\rm e}$tale groupoid $\mg_X$ is ample if and only if $\mg_X^{(0)}$ admits a basis of close and open subsets (see \cite[Proposition 4.1]{exel}).

A subset $V\subseteq X$ is \emph{invariant} if for any $g\in G$, we have ${\phi_g}|_V\subseteq V$, equivalently, $\phi_{g^{-1}}(X_g\cap V)\subseteq X_{g^{-1}}\cap V$. If, in addition, $V$ is open, then the restriction $\phi|_V=(\phi_g, X_g\cap V, V)_{g\in G}$ is a partial action of $G$ on $V$.

We observe that there is a one-to one correspondence between the (open) invariant subsets of $X$ and the (open) invariant subsets of $\mg_X^{(0)}$, assigning $V\subseteq X$ to $\varepsilon\times V \subseteq \mg_X^{(0)}$. This gives a subgroupoid $\mg_V=\bigcup_{g\in G}g\times V_g$ of $\mg_X$ with $V_g=V\cap X_g$. 

The canonical map $\mg_X\xra G, (g, x)\mapsto g$ is a continuous cocycle which makes $\mg_X$ a $G$-graded groupoid. Thus the  Steinberg algebra $A_R(\mg_X)$ is a $G$-graded
algebra with homogeneous components
\[
    A_{R}(\mg_X)_{g} = \{f\in A_{R}(\mg)\mid \supp(f)\subseteq g\times X_g\}.\]

Recall that $\mg_X^{(h)}$ denotes the collection of graded compact open bisections of $\mg_X$ which forms a $G$-graded inverse semigroup. Let $U\sub X$ be an open invariant subset. By Subsection \ref{subsection22}, there is a partial action $\pi=(\pi_B, U_B, U)_{B\in \mg_X^{(h)}}$ of the inverse semigroup $\mg_X^{(h)}$ on $U$ with $U_B=r(B)\cap U$. We still denote by $\pi$ the induced partial action $(\pi_B, C_R(U_B), C_R(U))_{B\in\mg_X^{(h)}}$ of the inverse semigroup $\mg_X^{(h)}$ on the algebra $C_R(U)$.

In order to use Theorem~\ref{thmpsisr} to express a partial skew group ring as a Steinberg algebra,  we need Exel's inverse semigroup $S(G)$ associated to a group $G$~\cite{exel1998}. We recall the construction here. 

\begin{deff} Let $G$ be a group with unit $\varepsilon$. We define  $S(G)$ to be the semigroup generated by $\{[g]\;|\; g\in G\}$ subject to the following relations: for $g, h\in G$,
\begin{itemize}

\item [(i)] $[g^{-1}][g][h] = [g^{-1}][gh]$, 

\item[(ii)] $[g][h][h^{-1}] = [gh][h^{-1}]$,

\item[(iii)] $[g][\varepsilon] = [g]$.
\end{itemize}
\end{deff}
Observe that $[\varepsilon][g]=[gg^{-1}][g]=[g][g^{-1}][g]=[g][g^{-1}g]=[g][\v]=[g]$. Then $S(G)$ is a semigroup with unit $[\varepsilon]$. 

The following is a universal property of $S(G)$.

\begin{lemma}\cite[Proposition 2.2]{exel1998} Given a semigroup $\SS$, a group $G$, and a map $f : G\xra \SS$ satisfying for $g, h\in G$
\begin{itemize}
\item[(i)] $f(g^{-1})f(g)f(h) = f(g^{-1})f(gh)$, 
\item[(ii)] $f(g)f(h)f(h^{-1}) = f(gh)f(h^{-1})$,
\item[(iii)] $f(g)f(\v) = f(g)$,
\end{itemize}
there exists a unique homomorphism $\overline f : S(G) \xra\SS$ such that $\overline f([g]) = f(g)$.
\end{lemma}

We recall some facts about $S(G)$, referring the reader to \cite[\S 2]{exel1998} for more details. For $g \in G$, define $\v_g = [g][g^{-1}]$. It is not difficult to see that $\v_g = \v_g^2$ and $[h]\v_g = \v_{hg}[h]$, for $g,h \in G$. Another interesting property is that any element $s\in S(G)$ admits a decomposition
$$s=\v_{l_1}\cdots \v_{l_n}[g],$$
where $n\geq 0$ and $l_1,\cdots, l_n, g\in G$. This decomposition is unique, except for the order of the $l_j$. We will call it the \emph{standard form} of $s$. Take an idempotent $e = \v_{e_1}\cdots \v_{e_n} [j] \in S(G)$, with $e_1,\cdots, e_n, j \in G$. By the uniqueness of decomposition of $S(G)$, $j$ must be the unit of the group. So $e=\v_{e_1}\cdots \v_{e_n}$. 

\begin{rmk} \label{rmkorder}For $r=\v_{r_1}\cdots \v_{r_n}[h]$ and $s=\v_{i_1}\cdots \v_{i_m}[g]$ in $S(G)$, if $s\leq r$ we have that $s=rf$, for some $f \in S(G)$ an idempotent. Then $f =\v_{f_1} \cdots \v_{f_k}$ and we have that:
$$s=\v_{i_1} \cdots \v_{i_m} [g] = \v_{r_1} \cdots \v_{r_n} [h]\v_{f_1} \cdots \v_{f_k} = \v_{r_1} \cdots \v_{r_n} \v_{hf_1} \cdots \v_{hf_k} [h].$$
By the uniqueness of the decomposition in $S(G)$, it follows that $g = h$ and ${i_1,\cdots,i_m} = {r_1,\cdots,r_n,hf_1,\cdots,hf_k}$. So, the difference between $s$ and $r$ are some $\v$'s.
\end{rmk}

 For every group $G$ and any set $X$, there is a one-to-one correspondence between partial actions of $G$ on $X$ and actions of $S(G)$ on $X$ (see \cite[Theorem 4.2]{exel1998}). Similarly, for an algebra $A$, there is a bijection between the partial actions of $G$ on $A$ and the actions of $S(G)$ on $A$ (see \cite[Theorem 2.9]{exelvieira}).

We recall that $\phi=(\phi_g, X_g, X)_{g\in G}$ induces a partial action of $G$ on $C_R(X)$, still denoted by $\phi$. For each $g\in G$, $\phi_g:C_R(X_{g^{-1}})\xra C_R(X_g)$ is given  by $\phi_g(f)=f\circ \phi_{g^{-1}}$, which is an $R$-isomorphism with the inverse $\phi_{g^{-1}}:C_R(X_{g})\xra C_R(X_{g^{-1}})$. 
 The ring $C_R(X)$ is idempotent and thus by \cite[Corollary 3.2]{de}, $C_R(X)\rtimes_{\phi}G$ is associated. We observe that $C_R(X)\rtimes_{\phi}G$ is a $G$-graded algebra where  $f_g\d_g$ for $f_g\in C_R(X_g)$ are homogeneous elements of degree $g$. 

For an open invariant subset $U$ of $X$, we have the restriction partial action $\phi|_U=(\phi_g, U_g, U)_{g\in G}$ with $U_g=X_g\cap U$. Then we have a partial action $(\phi_g, C_R(U_g), C_R(U))_{g\in G}$ of $G$ on an algebra $C_R(U)$, which is still denoted by $\phi$. There is a partial action $\phi'=(\phi'_s, E_s, C_R(U))_{s\in S(G)}$ of the inverse semigroup $S(G)$ on $C_R(U)$. For $s=\v_{i_1}\cdots\v_{i_n}[g]$ with $n\geq 0$ and $i_1, \cdots, i_n, g\in G$, $E_s=C_R(U_g)\cap C_R(U_{i_1})\cap\cdots\cap C_R(U_{i_n})$, $E_{s^*}=C_R(U_{g^{-1}})\cap C_R(U_{g^{-1}i_1})\cap\cdots\cap C_R(U_{g^{-1}i_n})$ and $\phi'_s:E_{s^*}\xra E_s$ is the restriction of $\phi_g:C_R(U_{g^{-1}})\xra C_R(U_g)$ to $E_{s^*}$ (refer to \cite[\S 2]{exelvieira} for details). Observe that the partial skew inverse semigroup ring $C_R(U)\rtimes_{\phi'}S(G)$ is a $G$-graded algebra. In fact, in this case $\LL=\{\sum_{s\in S(G)}a_s\d_s\;|\; a_s\in E_s\}$. The element $a\d_s\in \LL$ is of degree $g$ with $s=\v_{i_1}\cdots\v_{i_n}[g]$. By Remark \ref{rmkorder}, $a\d_s-a\d_r$ with $s,r\in S(G)$, $s\leq r$ and $a\in E_s$ is a homogeneous element and the quotient algebra $C_R(U)\rtimes_{\phi'}S(G)$ of $\LL$ is $G$-graded.

Note that in $C_R(U)\rtimes_{\phi'}S(G)$ \begin{equation}\label{property}\begin{split} 
 a\d_{[g][h]}=a\d_{[gh]} \text{~for~} a\in E_{[g][h]}, ~~~~
 a\d_{\v_{i_1}\cdots\v_{i_n}[g]}=a\d_{[g]} \text{~for~} a\in E_{\v_{i_1}\cdots\v_{i_n}[g]}.
\end{split}\end{equation} 
To see that $a \d_{[g][h]}=a\d_{[gh]}$ for $a\in E_{[g][h]}$, note that  $[g][h]=[g][h][h^{-1}][h]=[gh][h^{-1}][h]\leq [gh]$. 

We are in a position to prove the main result of this section. 
%We show that the Steinberg algebra $A_R(\mg_X)$ is graded isomorphic to the partial skew group ring.

\begin{prop}\label{propss}
Suppose that $X$ is a locally compact Hausdorff topological space which is totally disconnected and $G$ a discrete group. Let $\phi=(\phi_g, X_g, X)_{g\in G}$ be a partial action of $G$ on $X$ where $X_g$'s are clopen subsets of $X$. Let
 $U$ be a clopen invariant of $X$. Then there are $G$-graded $R$-algebra isomorphisms  
\begin{equation}\label{isoes} A_R(\mg_U) \cong_{\gr}C_R(U)\rtimes_{\pi}\mg_X^{(h)}\cong_{\gr} C_R(U) \rtimes_{\phi'} S(G)\cong_{\gr} C_R(U) \rtimes_\phi G.
\end{equation}In particular, there are $G$-graded $R$-algebra isomorphisms$$ A_R(\mg_X) \cong_{\gr}C_R(X)\rtimes_{\pi}\mg_X^{(h)}\cong_{\gr} C_R(X) \rtimes_{\phi'} S(G)\cong_{\gr} C_R(X) \rtimes_\phi G.$$\end{prop}

\begin{proof} The first isomorphism in \eqref{isoes} follows from Theorem \ref{thmpsisr}. By \cite[Theorem 3.7]{exelvieira}, 
\begin{align*}
\varphi: C_R(U)\rtimes_{\phi}G&\longrightarrow C_R(U)\rtimes_{\phi'}S(G)\\
a\d_g&\longmapsto a\d_{[g]}
\end{align*}
 is an isomorphism of algebras. It is evident that $\varphi$ preserves the grading. Then the last isomorphism in \eqref{isoes} holds.  

It remains to prove that $C_R(U) \rtimes_{\phi'} S(G)\cong_{\gr} C_R(U)\rtimes_{\pi}\mg_X^{(h)}$. Define $$\Psi: C_R(U) \rtimes_{\phi'} S(G)\xra C_R(U)\rtimes_{\pi}\mg_X^{(h)}$$ such that $\Psi(a_s\d_s)=a_s\d_{g\times U_g}$ for $a\in E_s$ and $s=\v_{i_1}\cdots\v_{i_n}[g]$ with $n\geq 0$ and $i_1, \cdots, i_n, g\in G$. Observe that $g\times U_g$ is a graded compact open bisection of $\mg_X$. Then by Remark \ref{rmkorder} $\Psi$ is well defined. We claim that $\Psi$ is a homomorphism of $G$-graded algebras. For another element $a_t\d_t$ in $C_R(U) \rtimes_{\phi'} S(G)$, we write $t=\v_{l_1}\cdots\v_{l_m}[h]$ with $m\geq 0$ and $l_1, \cdots, l_m, h\in G$. On one hand, we have 
\begin{equation*} 
\Psi\big((a_s\d_s)(a_t\d_t)\big)
=\Psi\big((\phi'_s(\phi'_{s^{-1}}a_s)a_t)\d_{st}\big)
=\Psi\big((\phi_g(\phi_{g^{-1}}a_s)a_t)\d_{\v_{i_1}\cdots\v_{i_n}\v_{gl_1}\cdots\v_{gl_m}[gh]}\big)
=(\phi_g(\phi_{g^{-1}}a_s)a_t)\d_{gh\times U_{gh}}.
\end{equation*} 
On the other hand, we have 
\begin{equation} \label{second}
\begin{split}
\Psi(a_s\d_s)\Psi(a_t\d_t)&=(a_s\d_{g\times U_g})(a_t\d_{h\times U_h})\\
&=(\phi_g(\phi_{g^{-1}}a_s)a_t)\d_{(g\times U_g \cdot h\times U_h)}\\
&=(\phi_g(\phi_{g^{-1}}a_s)a_t)\d_{gh\times \big(U_g\cap \phi_g(U_h\cap U_{g^{-1}})\big)}\\
&=(\phi_g(\phi_{g^{-1}}a_s)a_t)\d_{gh\times U_{gh}}.
\end{split}
\end{equation} Here, the last equality in \eqref{second} holds, since $gh\times \big(U_g\cap \phi_g(U_h\cap U_{g^{-1}})\big)\sub gh\times U_{gh}$. It follows that $\Psi\big((a_s\d_s)(a_t\d_t)\big)
=\Psi(a_s\d_s)\Psi(a_t\d_t)$. It is evident that $\Psi$ preserves the grading. Thus the proof for the claim is completed.

The map $\Psi$ is surjective as the following argument shows. Take an element $a_B\d_B\in C_R(U)\rtimes_{\pi}\mg_X^{(h)}$ with $B\in \mg_X^{(h)}$. We write $B=g\times V$ for some $g\in G$ and $V\sub U_g$. Then we have $a_B\in C_R(V)=C_R(U_B)$, and $a_B\in C_R(U_g)$, since $V\sub U_g$. It follows that $a_B\d_B=a_B\d_{g\times V}=a_B\d_{g\times U_g}=\Psi(a_B\d_{[g]})$ in $C_R(U)\rtimes_{\pi}\mg_X^{(h)}$.

To prove the injectivity of $\Psi$, we define $\Theta: C_R(U)\rtimes_{\pi}\mg_X^{(h)}\xra C_R(U)\rtimes_{\phi'}S(G)$ such that $\Theta(\sum_{i=1}^na_{B_i}\d_{B_i})=\sum_{i=1}^na_{B_i}\d_{[g_i]}$ with $B_i\in \mg_X^{(h)}$ such that  $B_i=g_i\times V_i$ for some $g_i\in G$ and $V_i\sub U_{g_i}$. Observe that $a_{B_i}\in C_R(U_{B_i})=C_R(V_i)\sub C_R(U_{g_i})$. If $B\leq C$ in $\mg_X^{(h)}$ and $a_B\in C_R(U_B)$, then $B$ and $C$ are contained in the same $g\times U_g$ for some $g\in G$. We have $\Theta(a_B\d_B-a_B\d_C)=a_B\d_{[g]}-a_B\d_{[g]}=0$. Thus $\Theta$ is well defined. We can directly check that $\Theta\circ \Psi=\id_{C_R(U)\rtimes_{\phi'}S(G)}$, implying that $\Psi$ is injective. 
\end{proof}

\begin{comment}
When $X$ is a locally compact Hausdorff topological space which is totally disconnected, the product topology $G\times X$ also  has a basis of compact open subsets. Then  $$\big\{\{g\}\times B\;|\; g\in G, B\subseteq X_g\text{~with~} B \text{~a compact open basis subset of~} X\big\}$$ is a basis of compact open bisections of the groupoid $\mg$.

Let $\phi=(\phi_g, X_g, X)_{g\in G}$ be a partial action of $G$ on $X$. We show that $\phi$ induces a partial action of $G$ on $C_R(X)$. For each $g\in G$ with $X_g\neq \emptyset$, let $C_R(X_g)$ be the $R$-algebra of locally constant functions with compact support from $X_g$ to $R$. Observe that $C_R(X_g)$ may be identified with the subsets of the functions in $C_R(X)$ that vanishes outside of $X_g$. For $g\in G$ with $X_g=\emptyset$, let $C_R(X_g)$ be the subset of $C_R(X)$ that contains only the null function. Notice that each $C_R(X_g)$ is an ideal of the $R$-algebra $C_R(X)$. 

For each $g\in G$,  define $\phi_g:C_R(X_{g^{-1}})\xra C_R(X_g)$ by $\phi_g(f)=f\circ \phi_{g^{-1}}$, which is an $R$-isomorphism with the inverse $\phi_{g^{-1}}:C_R(X_{g})\xra C_R(X_{g^{-1}})$. We have that $(\phi_g, C_R(X_g), C_R(X))_{g\in G}$ is a partial action of $G$ on the $R$-algebra  $C_R(X)$, which is still denoted by $\phi$. 
\end{comment}

\begin{cor} \label{corvon}Suppose that $X$ is a locally compact Hausdorff topological space which is totally disconnected and $G$ a discrete group. Let $\phi=(\phi_g, X_g, X)_{g\in G}$ be a partial action of $G$ on $X$ and $K$ a field. Then the Steinberg algebra $A_K(\mg_X)$ is $G$-graded von Neumann regular. 
\end{cor} 

\begin{proof} For $f\in C_K(X_g)$, define $f'\in C_K(X_g)$ by $f'(x)=f(x)^{-1}$ if $x\in \supp(f)$ and $f'(x)=0$ otherwise.  Note that $f.f'=1|_{\supp(f)}$. Thus $ff'f=f$. This shows that $C_K(X_g)$ is a von Neumann regular ring for each $g\in G$. By Propostion \ref{propss}, we have the graded isomorphism $A_K(\mg_X) \cong_{\gr} C_K(X) \rtimes_\phi G$. The consequence follows directly  by Lemma \ref{lemmavon}. 
\end{proof}

\subsection{Graded ideals of partial skew group rings}\label{subsection32}

Let $X$ be a locally compact Hausdorff topological space which is totally disconnected  and let $G$ be a discrete group. In this subsection, we consider a partial action $\phi=(\phi_g, X_g, X)_{g\in G}$ and the induced partial action $\phi=(\phi_g, C_K(X_g), C_K(X))_{g\in G}$ of $G$ on the algebra $C_K(X)$ and the partial skew group ring $C_K(X)\rtimes_{\phi}G$ associated to it. Throughout this subsection, $K$ is a field and $X_g$'s are open subsets of $X$.  We  prove that there is a one-to-one correspondence between open invariant subsets of $X$ and graded ideals of the partial skew group ring $C_K(X)\rtimes_{\phi}G$.

%We obtain a one-to-one correspondence between the graded ideals of the partial skew group ring $C_K(X)\rtimes_{\phi}G$ and the open invariant subsets of $X$.

Recall that an ideal $I$ of a $G$-graded ring $A=\bigoplus_{g\in G} A_g$ is a graded ideal if $I=\bigoplus_{g\in G} I_g$, where $I_g=I\bigcap A_g$.

\begin{lemma} \label{gradedideal} Suppose that $\phi=(\phi_g, X_g, X)_{g\in G}$ is a partial action of a discrete group $G$ on a topological space $X$. Let $V$ be an open invariant subset of $X$. Then $C_K(V)\rtimes_{\phi |V}G$ is a graded ideal of $C_K(X)\rtimes_{\phi}G$.
\end{lemma}

\begin {proof} For $a_g\in C_K(V\cap X_g)$ and $a_h\in C_K(X_h)$, we have $\phi_g(\phi_{g^{-1}}(a_g)a_h)\in C_K(V\cap X_{gh})$, since $\phi_{g^{-1}}(a_g)a_h\in C_K(V\cap X_{g^{-1}})\cap C_K(V\cap X_{h})$ and $\phi |V$ is a partial action. It follows that $a_g\delta_g\cdot a_h\delta_h=\phi_g(\phi_{g^{-1}}(a_g)a_h)\d_{gh}\in C_K(V)\rtimes_{\phi |V}G$, when $a_g\in C_K(V\cap X_g)$ and $a_h\in C_K(X_g)$. Similarly, we have $a_h\delta_h\cdot a_g\delta_g\in C_K(V)\rtimes_{\phi |V}G$, when $a_g\in C_K(V\cap X_g)$ and $a_h\in C_K(X_g)$. Thus $C_K(V)\rtimes_{\phi |V}G$ is an ideal of $C_K(X)\rtimes_{\phi}G$. It is evident that $C_K(V)\rtimes_{\phi |V}G=\bigoplus_{g\in G}(C_K(V)\rtimes_{\phi |V}G)\bigcap (C_K(X)\rtimes_{\phi}G)_g$. Thus $C_K(V)\rtimes_{\phi |V}G$ is a graded ideal of $C_K(X)\rtimes_{\phi}G$. 
\end{proof}

Suppose that $I$ is an ideal of the partial skew group ring $C_K(X)\rtimes_{\phi}G$. Set \begin{equation}
\label{definv}
V_I=\bigcup_{f\d_{\varepsilon}\in I}\supp (f).
\end{equation} 

\begin{lemma}\label{checkinv} 
%Let $X$ be a topology with a basis of compact open subsets and $G$ a discrete group. 
Suppose that $\phi=(\phi_g, X_g, X)_{g\in G}$ is a partial action of $G$ on $X$. Let $I$ be an ideal of the partial skew group ring $C_K(X)\rtimes_{\phi}G$. Then the above set $V_I$ is an open invariant subset of $X$. 
\end{lemma}

\begin{proof}
The set $V_I$ is open, since $\supp(f)$ is open for each $f$. To prove that $V_I$ is invariant, take any element $x\in V_I\cap X_g$ for $g\in G$. We need to show that $\phi_{g^{-1}}(x)\in V_I$. Since $x\in V_I$, there exists $f\in C_K(X)$ such that $f\d_{\varepsilon}\in I$ and $x\in \supp(f)$. We write $f=\sum_{i=1}^nr_i1_{D_i}$, where $D_i$ are compact open mutually disjoint subsets of $X$. We assume that $x\in D_1$. Then there exists a compact open set $C\sub D_1\cap X_g$ such that $x\in C$. We observe that \begin{equation*}
\begin{split}
1_{\phi_{g^{-1}}(C)}\d_{g^{-1}}\cdot f\d_{\varepsilon}
&=\sum_{i=1}^n 1_{\phi_{g^{-1}}(C)}\d_{g^{-1}}\cdot r_i1_{D_i}\d_{\varepsilon}\\
&=\sum_{i=1}^nr_i \phi_{g^{-1}}\big(\phi_g(1_{\phi_{g^{-1}}(C)})1_{D_i}\big)\d_{g^{-1}}\\
&=\sum_{i=1}^nr_i 1_{\phi_{g^{-1}}(C)}\phi_{g^{-1}}(1_{D_i})\d_{g^{-1}}\\
&=r_11_{\phi_{g^{-1}}(C)}\d_{g^{-1}}\in I.
\end{split}
\end{equation*} Then we have $1_{\phi_{g^{-1}}(C)}\d_{\varepsilon}=\phi_{g^{-1}}\big(\phi_g(1_{\phi_{g^{-1}}(C)})1_C\big)\d_{\varepsilon}=r_11_{\phi_{g^{-1}}(C)}\d_{g^{-1}}\cdot r_1^{-1}1_C\d_{g}\in I$. We have $\phi_{g^{-1}}(C)\subseteq V_I$, implying $\phi_{g^{-1}}(x)\in V_I$.
\end{proof}

\begin{lemma} \label{idealcontained} Let $I$ be an ideal of the partial skew group ring $C_K(X)\rtimes_{\phi}G$. Then we have $C_K(V_I)\rtimes_{\phi |V_I}G\sub I$.  Furthermore, $C_K(V_I)\rtimes_{\phi |V_I}G$ is the maximal graded ideal of $C_K(X)\rtimes_{\phi}G$ which is contained in $I$. 
\end{lemma}

\begin{proof} By Lemma \ref{gradedideal}, $C_K(V_I)\rtimes_{\phi |V_I}G$ is a graded ideal of $C_K(X)\rtimes_{\phi}G$. To prove that $C_K(V_I)\rtimes_{\phi |V_I}G\sub I$, it suffices to prove that $f_g\d_g\in I$ when $f_g\in C_K(V_I\cap X_g)$ for all $g\in G$. We first prove that $1_B\d_{\varepsilon}\in I$ for any compact open set $B\sub V_I$. Take any element $x\in B$. There exists $a_{\varepsilon}\in C_K(X)$ such that  $a_{\varepsilon}\d_{\varepsilon}\in I$ and $x\in \supp (a_{\varepsilon})$. Write $a_{\varepsilon}=\sum_{k=1}^l b_k1_{C_k}$, where $C_k$ are mutually disjoint compact open sets. The element $x$ belongs to one set among $C_1, \cdots, C_l$, which is denoted by $C_x$. Then we have $B\sub\cup_{x\in B}C_x=\cup_{s=1}^mC_{x_s}$, since $B$ is compact. It follows that $B=\bigcup_{s=1}^m (B\cap C_{x_s})$. Set $B_s=B\cap C_{x_s}$. By inclusion-exclusion principle, we have 
\begin{equation}
\label{inclusion1}
1_{\cup_{s=1}^mB_{s}}=\sum_{1\leq k\leq m}(-1)^{k+1}\sum_{\{i_1, \cdots, i_k\}\sub\{1, \cdots, m\}}1_{B_{i_1}\cap\cdots\cap B_{i_k}}.
\end{equation}
%Then $$\sum_{s=1}^m1_{B_{s}}=1_{\cup_{s=1}^mB_{s}}+\sum_{n< m}n1_{B_1\cap \cdots \cap B_n}.$$ 
We observe that $1_{C_{x_s}}\d_{\varepsilon}\in I$ for all $s$, since $a_{\varepsilon}\d_{\varepsilon}\in I$. So $1_{B_s}\d_{\varepsilon}=1_B\d_{\varepsilon}\cdot 1_{C_{x_s}}\d_{\varepsilon}\in I$ for each $s$. By \eqref{inclusion1}, $1_B\d_{\varepsilon}=1_{\cup_{s=1}^mB_{s}}\d_{\varepsilon}\in I$, since $1_{B_{i_1}\cap\cdots\cap B_{i_k}}\d_{\varepsilon}\in I$ for any $1\leq k\leq m$ and $\{i_1,\cdots, i_k\}\sub \{1,\cdots, m\}$. For any $f\in C_K(V_I)$, we write $f=\sum_{j=1}^lk_j1_{D_j}$ with the $D_j$'s disjoint compact open subsets of $V_I$. Then $f\d_{\varepsilon}\in I$. Next for any $g\neq \varepsilon$ and $f_g\in C_K(V_I\cap X_g)$, we have  
\begin{equation}\label{eq}
f_g\d_g=f_g\d_{\varepsilon}\cdot 1_{\supp(f_g)}\d_g\in IA\sub I.
\end{equation}

It remains to prove $J\sub C_K(V_I)\rtimes_{\phi |V_I}G$ for any graded ideal $J$ of $C_K(X)\rtimes_{\phi}G$ which is contained in $I$. For any $1_B\d_{\varepsilon}\in J_{\v}$ with $B$ a compact open subset of $X$, we have $1_B\d_{\v}\in I$ and thus $B\sub V_I$ by the definition of $V_I$ given in \eqref{definv}. It follows that $1_B\d_{\varepsilon}\in (C_K(V_I)\rtimes_{\phi |V_I}G)_{\v}$ and $f\d_{\v}\in (C_K(V_I)\rtimes_{\phi |V_I}G)_{\v}$ for any $f\d_{\v}\in J_{\v}$. Similarly as \eqref{eq}, we have $f_g\d_g\in (C_K(V_I)\rtimes_{\phi |V_I}G)_{g}$ for any $g\in G$. So $J\sub C_K(V_I)\rtimes_{\phi |V_I}G$.
\end{proof}

\begin{prop}\label{gradedideals} Let $X$ be a topological space that has a basis of compact open sets, $G$ a discrete group, $\phi=(\phi_g, X_g, X)_{g\in G}$ a partial action of $G$ on $X$. Suppose that $K$ is a field. Then there is a one-to-one correspondence $$V\longmapsto C_K(V)\rtimes_{\phi |V}G$$ between the open invariant subsets of $X$ and the graded ideals of $C_K(X)\rtimes_{\phi}G$.
\end{prop}

\begin{proof} Let $\mathcal{O}$ be the set of open invariant subsets of $X$ and $\mathcal{I}$ the set of graded ideals of $C_K(X)\rtimes_{\phi}G$. We define a map $\Phi:\O\xra \I$ by $\Phi(V)=C_K(V)\rtimes_{\phi |V}G$ for $V\in \O$. By Lemma \ref{gradedideal} the map $\Phi$ is well defined. 

We prove that $\Phi$ is injective. For $V, V'\in\O$, if $C_K(V)\rtimes_{\phi |V}G=C_K(V')\rtimes_{\phi |V'}G$, then we have $C_K(V)=C_K(V')$, since $C_K(V)$ and $C_K(V')$ are the $\varepsilon$-th components of the graded ideals $C_K(V)\rtimes_{\phi |V}G$ and $C_K(V')\rtimes_{\phi |V'}G$. Since $V$ is open, we have $V=\cup_{i} B_i$ with $B_i$ compact open sets of $X$. For any $B_i$, we have $1_{B_i}\in C_K(V)=C_K(V')$. Then we have $B_i\subseteq V'$ for any $i$. Thus $V\subseteq V'$. Similarly, we have $V'\subseteq V$. Hence, $V=V'$.

It remains to prove that $\Phi$ is surjective. Take $I\in\I$. By Lemma \ref{checkinv}, we have an open invariant subset $V_I\sub X$. We prove that $I=C_K(V_I)\rtimes_{\phi |V_I}G$. By Lemma \ref{idealcontained} we have $C_K(V_I)\rtimes_{\phi |V_I}G\subseteq I$ and $C_K(V_I)\rtimes_{\phi |V_I}G$ is the maximal graded ideal of $C_K(X)\rtimes_{\phi}G$ which is contained in $I$. Since $I$ is a graded ideal of $C_K(X)\rtimes_{\phi}G$, we have $I=C_K(V_I)\rtimes_{\phi |V_I}G=\Phi(V_I)$.
\end{proof}

Let $\mg $ be a locally compact Hausdorff $\acute{\rm e}$tale groupoid. 
Recall that the \emph{isotropy group} at a unit $u$ of a groupoid $\mg$ is the group  $\Iso(u) =\{\g\in\mg \mid
d(\g)=r(\g)=u\}.$ The groupoid $\mg$ is called \emph{effective} if the interior of  
$\Iso(\mg)$ is $\mg^{(0)}$, where $\Iso(\mg)=\bigsqcup_{u\in\mg^{(0)}}\Iso(u)$. The groupoid $\mg$ is effective if and only if  for any nonempty open bisection $B$ with $B \cap \mg^{(0)} = \emptyset$, we have $B \not \subseteq \Iso(\mg)$ (see \cite[Lemma~3.1]{bcfs}). A groupoid $\mg$ is \emph{strongly effective} if for every nonempty closed
invariant subset $D$ of $\mg^{(0)}$, the groupoid $\mg|_{D}=d^{-1}(D)$ is effective.

Let $\mg$ be a $G$-graded ample Hausdorff groupoid  such that $\mg_\varepsilon$ is strongly effective. It was proved in \cite[Theorem 5.3]{cep} that the correspondence
$$U\longmapsto A_K(\mg|_U)$$
is an isomorphism from the lattice of open invariant subsets of $\mg^{(0)}$ to the lattice of graded ideals in $A_K (\mg)$.

Given a partial action $\phi=(\phi_g, X_g, X)_{g\in G}$ of $G$ on $X$, we have a $G$-graded groupoid $\mg_X=\bigcup_{g\in G}g\times X_g$ given in \eqref{groupoid}.  Observe that the groupoid $(\mg_X)_{\varepsilon}=\varepsilon\times X$ is strongly effective, because there is no open bisection $B$ satisfying $B\cap (\mg_X)_{\varepsilon}^{(0)}=\emptyset$.

For any open invariant subset $V$ of $X$, recall that $\mg_V=\bigcup_{g\in G}g\times V_g$ is a subgroupoid of $\mg_X$ where $V_g=V\cap X_g$.  

\begin{cor} Let $X$ be a locally compact Hausdorff topological space which is totally disconnected. Suppose that $\phi=(\phi_g, X_g, X)_{g\in G}$ is a partial action of a discrete group $G$ on $X$ where $X_g$'s are clopen subsets of $X$. Then there is a one-to-one correspondence $$V\longmapsto A_K(\mg_V)$$ between the open invariant subsets of $X$ and the graded ideals of the Steinberg algebra $A_K(\mg_X)$.
\end{cor}

\begin{proof} By Proposition \ref{propss} we have $A_K(\mg_V)\cong_{\rm gr} C_K(V)\rtimes_{\phi |V}G$. By Proposition \ref{gradedideals}, the statement follows immediately.
\end{proof}

\subsection{A partial group ring} \label{subsection33}

Given a group $G$, Exel constructed a $C^*$-algebra, called partial group $C^*$-algebra of $G$ \cite[Definition 6.4]{exel1998}. In this section, we define the algebraic version of the construction and realise it as partial action of $G$ on an appropriate set. 

%Then we describe the graded ideals of the algebra. 

Let $G$ be an abelian group, $\v$ be its identity and $R$ a commutative ring with unit. Consider the free $R$-algebra $A$ generated by symbols $P_E$ with $E$ a finite subset of $G$, subject to relations $$P_EP_F=P_{E\cup F}$$ for all possible choices of $E$ and $F$. Note that $A$ is a commutative $R$-algebra, $P_EP_E=P_E$ and $P_{\emptyset}$ is the identity of $A$, where $\emptyset$ is the empty set. 

Let $A_{\v}$ be the ideal of $A$ generated by $P_{\{\v\}}$, that is, $A_{\v}=P_{\{\v\}}A$. Then $A_{\v}$ consists of the sum of $P_E$ such that $E$ is a finite subset of $G$ and $\v\in E$.  

Define $D_g={\rm span}_R\{P_E\;|\; \v, g\in E\}$ for $g\in G$. Then $D_g$ is an ideal of $A_\v$. Define $$\a_g: D_{g^{-1}}\longrightarrow D_g$$ by $\a_g(P_E)=P_{gE}$ which is an isomorphism of ideals of $A_\v$. We have $\a=(\a_g, D_g, A_\v)_{g\in G}$ is a partial action of $G$ on algebra $A_\v$. We call $$P(G)=A_\v\rtimes_{\a}G$$ the \emph{partial group ring} of $G$. 

Next we realise $P(G)$ as a partial skew group ring of the form $C_R(Y)\rtimes_{\phi}G$, where $\phi$ is a partial action of $G$ on a topological space $Y$. Consider $$Y=\prod_{g\in G}\{0, 1\}.$$ Endowed with the product topology, $Y$ is a compact Hausdorff space. Each element $x\in Y$ is considered as $\{x_g\}_{g\in G}$ with $x_g\in\{0,1\}$. Set $Y_\v$ to be the subset of $Y$ consisting of $\{x_g\}_{g\in G}$ with $x_\v=1$ and $$Y_g=\{x\in Y_\v\;|\; x_\v=1=x_g\}\sub Y_\v.$$  Here, $Y_\v$ is a subspace topology. Note that $Y_g$ is an open subset of $Y_\v$. Define $\phi_g: Y_{g^{-1}}\xra Y_g$ given by $$\phi_g(\{x_h\}_{h\in G})=\{x_{g^{-1}h}\}_{h\in G}.$$ Observe that $\phi_g$ is a homeomorphism. We can directly check that conditions (i) and (ii) of partial action of a discrete  group $G$ on $Y_\v$ hold. The condition (iii) of $G$ on $Y_\v$ holds since $G$ is an abelian group. So $\phi=(\phi_g, Y_g, Y_\v)_{g\in G}$ is a partial action of $G$ on the topological space $Y_\v$. 

\begin{comment}\note{have we used this groupoid?} We denote by $\mg(G)$ the associated groupoid of the partial action $\phi=(\phi_g, Y_g, Y_\v)_{g\in G}$. Then $\mg(G)=\bigcup_{g\in G}g\times Y_g$. Take $(\v, x)\in \mg(G)^{(0)}$. We have $x=\{x_g\}_{g\in G}\in Y_\v$. Set $$S_x=\{g\in G\;|\; x_g=1\}.$$ The isotropy group of $(\v, x)$ is 
\begin{equation*}
\begin{split}
\{(g, y)\in g\times Y_g\;|\; g\in G \text{~with~} y=x=\phi_{g^{-1}}(y)\}
=\{(g, x)\;|\; g\in G \text{~with~} g\in S_x \text{~and~} g^{-1}S_x=S_x\}\sub \{(g, x)\;|\; g\in S_x\}.
\end{split}\end{equation*} It follows that the isotropy group $\Iso(\v, x)=\{(\v, x)\}$ if and only if $s^{-1}S_x\neq S_x$ for any $s\in S_x$ with $s\neq \v$.
\end{comment}

The partial action $\phi=(\phi_g, Y_g, Y_\v)_{g\in G}$ of $G$ on $Y_\v$ induces a partial action $\phi=(\phi_g, C_R(Y_g), C_R(Y_\v))_{g\in G}$ of $G$ on $C_R(Y_{\v})$.

For each $g\in G$, let $$Q_g: Y\longrightarrow R$$ be the $g$-th coordinate function, that is, $Q_g(\{x_h\}_{h\in G})=x_g\in R$. Then $Q_g\in C_R(Y)$, since $\supp(Q_g)=\{x\in Y\;|\; x_g=1\}$ is an open compact subset of $Y$. For a finite set $E\sub G$, let $Q_E=\prod_{g\in E}Q_g$. Since $Q_EQ_F=Q_{E\cup F}$, we have a homomorphism of algebras $$\Psi: A\longrightarrow C_R(Y)$$ such that $\Psi(P_E)=Q_E$. When $E=\emptyset$, $Q_{\emptyset}$ is the identity map of $Y$.

Parallel to \cite[Proposition 6.6]{exel1998}, we prove that $\Psi: A\longrightarrow C_R(Y)$ is an isomorphism of algebras. 

\begin{prop} Let $G$ be a group. Then $\Psi: A\longrightarrow C_R(Y)$ is an isomorphism of algebras which restricts to $D_g\cong C_R(Y_g)$ for $g\in G$.
\end{prop}

\begin{proof}  We first prove injectivity of $\Psi$. Suppose that there exists $a=\sum_{i=1}^n k_iP_{E_i}\in A$ with $k_i\in R$ and $E_i$ pairwise distinct finite subsets of $G$ such that $\Psi(a)=\sum_{i=1}^nk_i Q_{E_i}=0$. If $n=1$, then it is obvious that $k_1=0$, implying $a=0$. We may assume that $n\geq 2$. We assume that $E_1$ is the one among $E_1, \cdots, E_n$ which has least number of elements. Denote by $|E_i|$ the cardinality of $E_i$ (possibly $|E_1|=|E_j|$ for some $j\neq 1$ ). Set $X_2=\{E_2, \cdots, E_n\}$. Pick $E_{i_2}\in X_2$. There exists $g_2\in E_{i_2}$ but $g_2\notin E_1$. Otherwise, we have $E_{i_2}\sub E_1$ and $|E_{i_2}|\geq |E_1|$, so $E_{i_2}=E_1$, contradicting to the fact that $E_1$ and $E_{i_2}$ are distinct. We list all the sets $E^2_1, \cdots, E^2_{n_2}$ among $E_2,\cdots, E_n$ which contains $g_2$ for some positive integer $n_2$. Suppose that for $k\geq 2$ we have distinct elements $g_2,\cdots, g_k\notin E_1$, $X_l=X_{l-1}\setminus{\{E^{l-1}_1,\cdots, E^{l-1}_{n_{l-1}}\}}$ and that $E^l_1,\cdots, E^l_{n_l}\in  X_l$ are the list of all sets in $X_l$ containing $g_l$ for $2\leq l\leq k$. Here, $X_1=\{E_1, \cdots, E_n\}$, $n_1=1$ and $E_1=E_1^1$. Set $X_{k+1}=X_k\setminus \{E^k_1, \cdots, E^k_{n_k}\}$. If $X_{k+1}=\emptyset$, we stop here. Otherwise, pick $E_{i_{k+1}}\in X_{k+1}$. Similarly, there exists  $g_{k+1}\in E_{i_{k+1}}\setminus E_1$. Observe that $g_{k+1}\neq g_l$ for $2\leq l\leq k$. We list all the sets $E^{k+1}_1, \cdots, E^{k+1}_{n_{k+1}}$ for some positive integer which contain $g_{k+1}$. There exists a number $2\leq t\leq n$ such that $X_{t+1}=\emptyset$. We have distinct elements $g_2,\cdots, g_t\notin E_1$, $X_l=X_{l-1}\setminus{\{E^{l-1}_1,\cdots, E^{l-1}_{n_{l-1}}\}}$ and that $E^l_1,\cdots, E^l_{n_l}\in  X_l$ are the list of all sets in $X_l$ containing $g_l$ for $2\leq l\leq t$.

Now we take $x=\{x_g\}_{g\in G}\in Y$ such that $x_g=1$ for all $g\in E_1$, $x_{g_l}=0$ for $2\leq l\leq t$. Observe that $Q_{E_1}(x)=1$ and $Q_{E_j}(x)=0$ for all $j\neq 1$. It follows that $\Psi(a)=k_1=0$. Thus $a=\sum_{i=2}^nk_iP_{E_i}$ satisfies $\Psi(a)=0$. Using the above process repeatitly, we show that $k_i=0$ for $i=2,\cdots, n$. Hence $a=0$.

Next we prove surjectivity. It suffices to prove $1_B\in {\rm Im} \Psi$ for any compact open subset $B$ of $Y$. Observe that $B$ is a finite union of $\prod_{g\in G}Z_g$ with finitely many $g$ satisfying $Z_g\neq \{0, 1\}$. Using disjointification \cite[Remark 2.4]{cfst}, we only need to show that $1_{\prod_{g\in G}Z_g}\in {\rm Im} \Psi$ with finitely many $g$ satisfying $Z_g\neq \{0, 1\}$. Let $E'$ be the collection of all the elements $g\in G$ with $Z_g\neq \{0, 1\}$. If $Z_g=\{1\}$ for all $g\in E'$, then $1_{\prod_{g\in G}Z_g}=Q_{E'}\in {\rm Im} \Psi$. Otherwise we have a disjoint union $E'=E\cup F$ such that $E$ consists all $g\in E'$ with $Z_g=\{1\}$ and $F$ consists of all $g\in E'$ with $Z_g=\{0\}$. We claim that $1_{\prod_{g\in G}Z_g}=Q_E+\sum_{D\sub F, D\neq \emptyset}(-1)^{|D|}Q_{D\cup E}\in {\rm Im} \Psi$. To prove the claim, take $x=\{x_g\}_{g\in G}\in Y$. If $x\in \prod_{g\in G}Z_g$, we have $\big(Q_E+\sum_{D\sub F, D\neq \emptyset}(-1)^{|D|}Q_{D\cup E}\big)(x)=1$, since $Q_{D\cup E}(x)=0$ for each $D\sub F$ with $D\neq \emptyset$. If $x\notin \prod_{g\in G}Z_g$, we have two subcases. One subcase is that $x_g=1$ for all $g\in E$ (if E is not the empatyset) and there are $m\geq 1$ number $x_{h}=1$ for $h\in F$. It follows that 
\begin{equation}
\begin{split}
\big(Q_E+\sum_{D\sub F, D\neq \emptyset}(-1)^{|D|}Q_{D\cup E}\big)(x)
&=1+(-1){m \choose 1}+(-1)^2{m \choose 2}+\cdots+(-1)^{m}{m \choose m}=0.
\end{split}
\end{equation} Another subcase is that there exists $g\in E$ such that $x_g=0$. We have $\big(Q_E+\sum_{D\sub F, D\neq \emptyset}(-1)^{|D|}Q_{D\cup E}\big)(x)=0$. Observe that if $E=\emptyset$, the proof is true. Therefore, the proof for the claim is completed. 

We observe that for each $g\in G$ and $E$ a finite subset of $G$ with $g\in E$, we have $Q_E\in C_R(Y_g)$. Conversely, for $1_B\in C_R(Y_g)$, the preimage of $1_B$ under $\Psi$ belongs to $D_g$. Thus $\Psi$ restricts $D_g\cong C_R(Y_g)$ for each $g\in G$.
 \end{proof}

Take $g\in G$. We have the following commutative diagram 
\begin{equation*}
\xymatrix{ D_{g^{-1}} \ar[d]^{\a_g} \ar[r]^(.4){\Psi} &C_R(Y_{g^{-1}}) \ar[d]^{\phi_g}\\
D_g \ar[r]^(.4){\Psi}& C_R(Y_g).
}
\end{equation*} Indeed, for a finite subset $E$ of $G$ with $\v, g^{-1}\in E$, we have $$(\phi_g\circ \Psi)(P_E)(x)=(Q_E\circ\phi_{g^{-1}})(x)=Q_E(\{x_{gh}\}_{h\in G})=\prod_{h\in G}x_{gh}=Q_{gE}(x)=(\Psi\circ \a_g)(P_E)(x).$$ Thus we have $$P(G)=A_\v\rtimes_{\a}G\cong C_R(Y_{\v})\rtimes_{\phi}G.$$

It would be interesting if one can characterise the graded ideals of these algebras. 

\begin{comment}

Now we apply Proposition \ref{gradedideals} to describe the graded ideals of $P(G)$. Observe that there is a bijection between $Y_\v$ and the subsets of $G$ which contains $\v$, i.e., any subset of $G$ containing $\v$ represents an element of $Y_\v$ and vice versa.

Let $S\sub G$ be a subset of $G$ with $\v\in S$. Define $$V_S=\{s^{-1}S\;|\; s\in S\}$$ which is a subset of $Y_\v$. Here, $Y_\v$ is viewed as a set of subsets of $G$ containing $\v$. We claim that $V_S$ is an invariant subset of $Y_\v$. To prove the claim, for any $s\in S$, we have $\v=s^{-1}s\in s^{-1}S$. So $s^{-1}S\in Y_\v$ and $V_S$ is a subset of $Y_\v$. For any $g\in G$, suppose that $V_S\cap Y_g\neq \emptyset$. There exists $s_i\in S$ such that $s_i^{-1}S\in Y_g$. Then $g=s_i^{-1}s_j$ for some $s_j\in S$. We have $\phi_{g^{-1}}(s_i^{-1}S)=g^{-1}s_i^{-1}S=(s_i^{-1}s_j)^{-1}s_i^{-1}S=s_j^{-1}S\in V_S$.

On the other hand, suppose that $V$ is an invariant subset of $Y_\v$. Then $V=\bigcup_{S\in V}V_S$. Because for any $S\in V$, we have $\v\in S$ and $S=\v^{-1}S\in V_S$, implying $V\sub \bigcup_{S\in V}V_S$. Then $V=\bigcup_{S\in V}V_S$ follows from the fact that $s^{-1}S=\phi_{s^{-1}}(S)\in V$ for any $S\in V$ and $s\in S$.

Is $V_S$ open or not?

\end{comment}

\section{Applications to Leavitt path algebras}\label{section4}

In this section, we consider Leavitt path algebras over arbitrary graphs $E=(E^{0}, E^{1}, r, s)$ (see~\cite{abrams} for notations and construction). We apply Proposition \ref{propss} to obtain Gon\c{c}alves-Royer's Theorem which realises Leavitt path algebras as partial skew group rings (see  \cite[Theorem 3.3]{goncalvesroyer}). This allows us to describe a graph condition under which the associated Leavitt path algebra can be realised as a $\Z$-graded partial skew group ring. 

\begin{comment}
A directed graph $E$ is a tuple $(E^{0}, E^{1}, r, s)$, where $E^{0}$ and $E^{1}$ are
sets and $r,s$ are maps from $E^1$ to $E^0$. We think of each $e \in E^1$ as an arrow
pointing from $s(e)$ to $r(e)$. We use the convention that a (finite) path $p$ of length $n$ in $E$ is
a sequence $p=\a_{1}\a_{2}\cdots \a_{n}$ of edges $\a_{i}$ in $E$ such that
$r(\a_{i})=s(\a_{i+1})$ for $1\leq i\leq n-1$. In this case, we denote by $|p|=n$. We define $s(p) = s(\a_{1})$, and $r(p) =
r(\a_{n})$. If $s(p) = r(p)$, then $p$ is said to be closed. If $p$ is closed and
$s(\a_i) \neq s(\a_j)$ for $i\neq j$, then $p$ is called a cycle. An edge $\a$ is an exit
of a path $p=\a_1\cdots \a_{n}$ if there exists $i$ such that $s(\a)=s(\a_i)$ and
$\a\neq\a_i$.

A directed graph $E$ is said to be \emph{row-finite} if for each vertex $u\in E^{0}$,
there are at most finitely many edges in $s^{-1}(u)$. A vertex $u$ for which $s^{-1}(u)$
is empty is called a \emph{sink}, whereas $u\in E^{0}$ is called an \emph{infinite
emitter} if $s^{-1}(u)$ is infinite. If $u\in E^{0}$ is neither a sink nor an infinite
emitter, then it is called a \emph{regular vertex}.

\end{comment}

Recall that a group $G$ is called an \emph{ordered group} if its elements can be given a total ordering $\leq$ which is left and right invariant, meaning that $g \leq h$ implies $fg \leq fh$ and $gf\leq hf$ for all $f,g,h\in G$. We refer to the pair $(G, \leq)$ as the ordered group. As in previous sections $\varepsilon$ denotes the identity element of a group $G$. We call an element $c$ of an ordered group positive if $\varepsilon \leq c$ and $c\neq \varepsilon$. The set of positive elements in an ordered group $G$ is denoted by $G_+$. Recall that a function $w:E^1\rightarrow G$, induces a $G$-grading on $L_R(E)$, by defining $w(v)=\varepsilon$, $v\in E^0$,  and $w(e^*)=w(e)^{-1}$, $e\in E^1$. Here $L_R(E)$ is a Leavitt path algebra associated to the graph $E$ with coefficients in the commutative ring $R$.

\begin{lemma}\label{uniqueness} (Generalised graded uniqueness theorem) Let $E$ be a graph, $G$ an ordered group, $w: E^1\xra G_+$ and $R$ a commutative ring with identity. Suppose that $\pi: L_R(E)\xra A$ is a homomorphism of $G$-graded $R$-algebras such that $\pi(rv)\neq 0$ for all $v\in E^0$, and $r\in R$. Then $\pi$ is injective. 
\end{lemma}

\begin{proof}
Suppose that $x\in L_R(E)$ is a nonzero element satisfying $\pi(x)=0$. Using a similar proof as \cite[Proposition 3.1]{pbgm}, one can show that there exist $\a, \b\in L_R(E)$ such that $\a x\b\in Rv$ for some $v\in E^0$, or there exist a vertex $w\in  E^0$ and a cycle without exits $c$ based at $w$ such
that $\a x\b$ is a nonzero element in $wL_R(E)w$. If $\a x\b\in Rv$, then we have $\pi(\a x\b)=0$.  
This is a contradiction as $\pi(rv)\not= 0$. If $\a x\b=\sum_{i=-m}^nr_i c^i$ for $m,n \in \mathbb{N}$ and $r_i\in R$, then $w(c^i)\neq w(c^j)$ for $-m\leq i, j \leq n$ and $i\neq j$. Otherwise, $w(c^{|i-j|})=w(c)^{|i-j|}=\varepsilon$ which is a contradiction, because $w(c)\in G_{+}$ and an ordered group has no elements of finite order except the identity. We have $r_ic^i\in \Ker \pi$ for each $-m\leq i\leq n$, since $\a x \b\in \Ker \pi$ and $\Ker \pi$ is $G$-graded. Then we have $r_ic^ic^{-i}=r_iw\in \Ker \pi$. This is a contradiction again. Hence, $\pi$ is injective. 
\end{proof}

A Leavitt path algebra associated to an arbitrary graph can be realised as the Steinberg
algebra associated to a groupoid of boundary path space of the graph $E$ (cf.~\cite{cs}).  Suppose that $G$ is an ordered group and $w: E^1\xra G$ a function satisfying $w(e)\in G_+$ for each edge $e\in E^1$. Then Leavitt path algebra $L_R(E)$ is a $G$-graded algebra. We extend $w$ to $E^*$ by defining $w(v)=0$ and $w(\a_1\cdots\a_n)=w(\a_1)\cdots w(\a_n)$. We construct a slightly different groupoid associated to the graph $E$ in order to realise Leavitt path algebras as $G$-graded Steinberg algebras.

For a directed graph $E$, we denote by $E^{\infty}$ the set of
infinite paths in $E$ and by $E^{*}$ the set of finite paths in $E$.  Set
\begin{equation}
\label{space}
X := E^{\infty}\cup  \{\mu\in E^{*}  \mid   r(\mu) \text{ is not a regular vertex}\}.
\end{equation}
% Using Graded Uniqueness Theorem one can show that a Leavitt path algebra is $\Z$-graded isomorphic to a $\Z$-graded Steinberg algebra; see \cite[Example 3.2]{cs}. Here, Leavitt path algebras are naturally $\Z$-graded by the length of paths, i.e, $w:E^1\rightarrow \Z, w(e)=1$. 
and  \begin{equation}
\label{defgroupoid}
\mg_{E} := \big \{(\a x,w(\a)w(\b)^{-1}, \b x) \mid   \a, \b\in E^{*}, x\in X, r(\a)=r(\b)=s(x)\big \}.
\end{equation} We view each $(x, k, y) \in \mg_{E}$ as a morphism with range $x$ and source $y$ and the 
formulas $(x,g,y)(y,h,z)= (x, gh,z)$ and $(x,g,y)^{-1}= (y,g^{-1},x)$ define composition
and inverse maps on $\mg_{E}$ making it a groupoid with ${\mg_{E}}^{(0)}=\{(x, \varepsilon, x) \mid
x\in X\}$ which we identify with the set $X$. Here, $\varepsilon$ is the identity of $G$. 

Next, we describe a topology on $\mg_{E}$. For $\mu\in E^{*}$ define
\[
Z(\mu)= \{\mu x \mid x \in X, r(\mu)=s(x)\}\subseteq X.
\]
For $\mu\in E^{*}$ and a finite $F\subseteq s^{-1}(r(\mu))$, define
\[
Z(\mu\setminus F) = Z(\mu) \setminus \bigcup_{\a\in F} Z(\mu \a).
\] $X={\mg_{E}}^{(0)}$ is a locally compact Hausdorff space with the topology given by the basis $$\{Z(\mu \setminus F): \mu \in E^*, F \text{~is a finite subset of~} r(\mu)E^1\},$$ and each such $Z(\mu \setminus F)$ is compact and open (see \cite[Theorem 2.1]{we} and \cite[Theorem 2.2]{we}).

For $\mu,\nu\in E^{*}$ with $r(\mu)=r(\nu)$, and for a finite $F\subseteq s^{-1}(r(\mu))$, we define
\[
Z(\mu, \nu)=\{(\mu x, w(\mu)w(\nu)^{-1}, \nu x) \mid  x\in X, r(\mu)=s(x)\},
\]
and then
\[
Z((\mu, \nu)\setminus F) = Z(\mu, \nu) \setminus \bigcup_{\a\in F}Z(\mu\a, \nu\a).
\]
The sets $Z((\mu, \nu)\setminus F)$ constitute a basis of compact open bisections for a
topology under which $\mg_{E}$ is a Hausdorff ample groupoid (refer to \cite[\S2.3]{bcw}).

We have a continuous $1$-cocycle $\widetilde{w}:\mg_E\xra G$ such that $\widetilde{w}(x,g, y)=g$; compare with \cite[Lemma 2.3]{kp}. Thus the Steinberg algebra $A_R(\mg_E)$ is a $G$-graded
algebra (see~(\ref{hgboat})) with homogeneous components
\[
    A_{R}(\mg_E)_{g} = \{f\in A_{R}(\mg_E)\mid \supp(f)\subseteq \widetilde{w}^{-1}(g)\}.
\] Similarly as \cite[Example 3.2]{cs}, we have a homomorphism $\pi_E: L_R(E)\xra A_R(\mg_E)$ of $G$-graded algebras such that $\pi_E(\mu\nu^{*}-\sum_{\a\in F}\mu\a\a^{*}\nu^{*})=1_{Z((\mu,\nu)\setminus F)}$. By Lemma \ref{uniqueness}, $\pi_E$ is an injective homomorphism of $G$-graded algebras. Again as in \cite[Example 3.2]{cs}, $\pi_E:L_R(E)\xra A_R(\mg_E)$ is an isomorphism of $G$-graded algebras.

In order to describe a Leavitt path algebra $L_R(E)$ as a partial skew group ring, we assign a free group grading to $L_R(E)$ as follows.  Let $\F$ be the free group generated by $E^1$. The map $w:E^1\rightarrow \F, e\mapsto e$, with $w(e^*)=e^{-1}$ induces an $\F$-grading on $L_R(E)$.  Recall the set $X$ from (\ref{space}) and consider the subsets 
(\cite[\S 2]{goncalvesroyer})
 $$X_{ab^{-1}}=\{x\in X\;|\; x=ax', x'\in X\}$$ and bijective maps $$\theta_{ab^{-1}}:X_{ba^{-1}}\longrightarrow X_{ab^{-1}}$$ for $ab^{-1}\in\F$ with $a, b\in \bigcup_{n=1}^{\infty}E^n$ and $r(a)=r(b)$. We then have that $\theta=(\theta_c, X_c, X)_{c\in\F}$ is a partial action of $\F$ on the set $X$ given in \eqref{space}. For a  commutative ring $R$ with identity, we have the induced partial action $\theta=(\theta_c, C_R(X_c), C_R(X))_{c\in \F}$ of $\F$ on the algebra $C_R(X)$.

%Recall from \cite{goncalvesroyer} that Leavitt path algebras were realised as partial skew group rings. 

Now we apply Proposition \ref{propss} to obtain an $\F$-graded isomorphism between the Leavitt path algebras and the partial skew group ring, which implies that Leavitt path algebras are $\Z$-graded isomorphic to the partial skew group ring; compare with \cite[Theorem 3.3]{goncalvesroyer}.

\begin{cor}\label{corlpa}
Let $E$ be an arbitrary graph and $\F$ a free group generated by $E^1$. Then we have the isomorphism $L_R(E)\cong C_R(X)\rtimes_{\theta}\F$ of $\F$-graded algebras.
\end{cor}

\begin{proof} We prove that the groupoids $\mg_E$ and $\mg=\bigcup_{c\in \F}c\times X_c$ are $\F$-graded isomorphism. Define a map 
\begin{align*}
\mg_E &\longrightarrow \bigcup_{c \in \F} c \times X_c\\
(px,pq^{-1},qx)&\longmapsto (pq^{-1},px),  
\end{align*} which is an isomorphism of groupoids. It is evident that this isomorphism of groupoids preserves the grading and topology. Thus there is an induced  isomorphism $A_R(\mg_E) \cong A_R(\mg) $ of $\F$-graded algebras. The fact that $ L_R(E) \cong A_R(\mg_E)$ as $\F$-graded algebras along with Proposition~\ref{propss} gives the following isomorphisms
\[L_R(E) \cong A_R(\mg_E) \cong A_R(\mg) \cong  C_R(X) \rtimes_\theta \F.\qedhere\]
\end{proof}

 For $v\in E^0$, $X_v=\{x\in X\;|\; s(x)=v\}$. Since $s(v)=v$, then $v\in X_v$ if and only if $v$ is not a regular vertex. There is a partial action $(\a_p, D_p, D(X))_{p\in \F}$ of $\F$ on algebra $D(X)={\rm span}\{\{1_p\;|\; p\in \F\setminus{0}\}\cup \{1_v\;|\; v\in E^0\}\}$, where $D_p={\rm span}\{1_p1_q\;|\; q\in \F\}$ for each $p\in \F\setminus \{0\}$ and $\a_c:D_{c^{-1}}\xra D_c$ is given by $\a_c(f)=f\circ \theta_{c^{-1}}$ \cite[\S 2]{goncalvesroyer}. Here, span means the $R$-linear span, $1_p$ the characteristic function of the set $X_p$ and $1_v$ the characteristic function of the set $X_v$.

%We consider $X=\mg_E^{(0)}$ as a locally compact Hausdorff topological space with a basis of compact open sets. Then $\theta=(\theta_c, X_c, X)_{c\in\F}$ is a partial action of the discrete group $\F$ on the topological space $X$.  For a  commutative ring $R$ with identity, we have the induced partial action $\theta=(\theta_c, C_R(X_c), C_R(X))_{c\in \F}$ of $\F$ on the algebra $C_R(X)$. 

We observe that the two partial actions $(\theta_c, C_R(X_c), C_R(X))_{c\in \F}$ and $(\a_c, D_c, D(X))_{c\in \F}$ coincide, 
or equivalently $C_R(X)=D(X)$, $C_R(X_c)=D_c$ and $\theta_c=\a_c$ for $c\in\F$ with $X_c\neq \emptyset$. We first show that $C_R(X)=D(X)$. Observe that $D(X)\sub C_R(X)$, since $1_p, 1_v\in C_R(X)$ for $p\in \F\setminus \{0\}$ and $v\in E^0$. To prove that $C_R(X)\sub D(X)$, we show that $1_B\in D(X)$ for any compact open subset $B$ of $X$. We write $B=\bigcup_{i=1}^nZ(\mu_i\setminus F_i)$. By the disjointification in \cite[Example 3.2]{cs}, $B$ is a disjoint union of sets of the form $Z(\mu\setminus F)$. We only need to show that $1_{Z(\mu\setminus F)}\in D(X)$ for $\mu\in E^*$ and $F$ a finite subset of $s^{-1}(r(\mu))$. We have 
\begin{equation}\begin{split}1_{Z(\mu\setminus F)}=1_{Z(\mu) \setminus \bigcup_{\a\in F} Z(\mu \a)}=1_{Z(\mu)}-1_{Z(\mu)\cap (\bigcup_{\a\in F}Z(\mu\a))}
&=1_{Z(\mu)}-1_{Z(\mu)}1_{\bigcup_{\a\in F}Z(\mu\a))}\in D(X),
\end{split}
\end{equation} since $1_{Z(\mu)}\in D(X)$ and by inclusion-exclusion principle $1_{\bigcup_{\a\in F}Z(\mu\a))}\in D(X)$. It is evident that $C_R(X_c)=D_c$ and $\a_c=\theta_c: D_{c^{-1}}\xra D_c$ for $p\in\F\setminus \{0\}$.

There is a group homomorphism $\psi: \F\xra \Z$ given by $\psi(c)=m-n$ for $c\in\F$, where $m$ is the number of generators (elements of $E^1$) of $c$, and $n$ is the number of inverses of generators of $c$. Then the partial skew group ring $C_R(X)\rtimes_{\theta}\F$ is a $\Z$-graded algebra. Recall that $L_R(E)$ is naturally $\Z$-graded by the length of paths in $E$. It is easy to see that the isomorphism in Corollary \ref{corlpa} is also an isomorphism of $\Z$-graded algebras. However, $\psi: \F\xra \Z$ may can not induce a partial action of $\Z$ on $X$. \begin{example}Let $E$ be the following graph. 

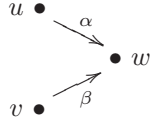
\begin{equation*}
\xymatrix@C=0.2cm@R=0.3cm{
               &      &   &u~\bullet\ar[drr]^{\alpha}    &     &\\
   & & &  &          &\bullet ~w\\
               &      &   &v ~\bullet\ar[urr]_{\beta}       &  &
             }
\end{equation*} Then the boundary path space is  $X=\{\a,\b,w\}$. There exist paths $\alpha$ and $\beta$ of length one terminating at the sink vertex $w$ which give rise to  $X_{\alpha\beta^{-1}}\neq X_{\beta\alpha^{-1}}$. Recall that there is a partial action $\phi=(\phi_c, X_c, X)_{c\in \F}$ of $\F$ on X, where $\F$ is the free group generated by $\{\a,\b\}$. Now by Lemma \ref{inducedpartial}, one can check that $\psi:\F\xra \Z$ can not induce a partial action of $\Z$ on $X$. 
\end{example}

A functor $F : \mathcal{C}\xra \mathcal{D}$ between two small categories $\mathcal{C}$ and $\mathcal{D}$ is called \emph{star injective} \cite[page 8]{batista}, if the map $$F |_{\star(x)} : \star(x) \longrightarrow \star(F(x))$$ is injective, where $\star(x) =\{f : x\xra y \text{~a morphism in~} \mathcal{C} \;|\; y \in \mathcal{C}\}$ for
every object $x \in\mathcal{C}$.

From now on, let $E$ be a graph and $w: E^1\xra \Z$ a function assigning each edge to $1$. Then by \eqref{defgroupoid} the groupoid $\mg_E$ is given by the standard 
$$\mg_E=\big \{(\a x,|\a|-|\b|, \b x) \mid   \a, \b\in E^{*}, x\in X, r(\a)=r(\b)=s(x)\big \}.$$ In this case, the 1-cocycle $\widetilde{w}:\mg_E\xra \Z$ is given by $\widetilde{w}(x, k, y)=k$ for $(x, k, y)\in \mg_E$.   

\begin{lemma}\label{iffstarinj}
Let $E$ be an arbitrary graph and $w: E^1\xra \Z$ a function assigning each edge to $1$. Then $\widetilde{w}:\mg_E\xra \Z$ is star injective if and only if $|r^{-1}_E(v)|\leq 1$ for each $v\in E^0$.
\end{lemma}

\begin{proof}
$\Rightarrow$: Suppose that there exists $w\in E^0$ such that $|r^{-1}_E(w)|\geq 2$. Then there exist at least two distinct edges $\a$ and $\b$ such that $r(\a)=r(\b)=w$. We have two cases. The first case is that $w$ is connected to a sink $z\in E^0$. Then there exists a path $p$ with $s(p)=w$ and $r(p)=z$. Then $\a p, \b p\in X$ and $\star(\a p)$ contains at least two elements $(\a p, 0, \a p)$ and $(\b p, 0, \a p)$. Since $\widetilde{w}(\a p, 0, \a p)=0=\widetilde{w}(\b p, 0, \a p)$, $\widetilde{w}: \star(\a p)\xra \Z$ is not injective. A contradiction. The other case is that $w$ is not connected to a sink. We have an infinite path $q$ starting at $w$. So $\a q, \b q\in X$. We obtain that $\widetilde{w}:\star(\a q)\xra \Z$ is not injective, again a contradiction. 

$\Leftarrow$: Suppose that $|r^{-1}_E(v)|\leq 1$ for each $v\in E^0$. For any $x\in \mg_E^{(0)}$, we have $\star(x)=\{(y, k, x)\in \mg_E \}$. Suppose that $(y_1, k_1, x)$ and $(y_2, k_2, x)$ both belong to $\star(x)$. We need to show that if $k_1=k_2$, then $y_1=y_2$. We write $x=v_1x_1=v_2x_2, y_1=u_1x_1$ and $y_2=u_2x_2$ with $u_1,u_2, v_1,v_2$ paths in $E$. Since $r(u_1)=r(v_1)$, we have $u_1=\d v_1$ or $v_1=fu_1$ for $\d, f$ paths in $E$. Similarly, either $u_2=sv_2$ or $v_2=tu_2$ for $s, t$ paths in $E$. If $u_1=\d v_1$, then $u_2=sv_2$ because $|u_1|-|v_1|=k_1=k_2=|u_2|-|v_2|$. It follows that $|\d|=|s|$ and $r(\d)=s(v_1)=s(x)=s(v_2)=r(s)$. Thus $\d=s$ and  $$y_1=u_1x_1=\d v_1x_1=\d x=sx=sv_2x_2=u_2x_2=y_2.$$ If $v_1=fu_1$, then we have $v_2=tu_2$. We have $|f|=|v_1|-|u_1|=|v_2|-|u_2|=|t|$ and $fu_1x_1=v_1x_1=x=v_2x_2
=tu_2x_2$. Thus $f=t$ and $y_1=u_1x_1=u_2x_2=y_2$.
\end{proof}

Recall that given a partial action $\phi=(\phi_g, X_g, X)_{g\in G}$ of a group $G$ on a set $X$, we have the associated groupoid $\mg_X=\bigcup_{g\in G}g\times X_g$. The set of objects of $\mg_X^{(0)}$ is identified with $X$. For $x\in \mg_X^{(0)}$, we have $\star(x)=\{(g, \phi_g(x))\;|\; g\in G\}$. Observe  that the projection $\pi: \mg_X \xra G$, given by $\pi(g, x) = g$ is a functor between the category $\mg_X$ to the category $G$ which is viewed as a small category with one object. The functor $\pi:\mg_X\xra G$ is star injective. Conversely, recall from \cite[page 8]{batista} that given a groupoid $\mg$ with a star injective functor $F:\mg\xra G$ (viewing $G$ as a small category with one object), we can associate a partial action of $G$ on $\mg^{(0)}$. Indeed, set $X=\mg^{(0)}$ and for each $g \in G$ set
$$X_g=\{x\in \mg^{(0)}\; |\; \exists ~\g\in\mg, r(\g)=x, F(\g)=g\}.$$ Define the maps $\phi_{g^{-1}} : X_{g}\xra X_{g^{-1}}$ as
$\phi_{g^{-1}}(x) = d(\g)$ such that $r(\g) = x, F(\g) = g$. This map is well defined because the functor $F$ is star injective. Then we have a partial action $\phi=(\phi_g, X_g, X)_{g\in G}$ of $G$ on $\mg^{(0)}$. The given groupoid $\mg$ with the star injective functor $F:\mg\xra G$ is isomorphic to the associated groupoid $\mg_X=\bigcup_{g\in G}g\times X_g$ with $X=\mg^{(0)}$. Define \begin{equation}\label{isogroupoid}
\eta:\mg\longrightarrow \mg_X
\end{equation} by $\eta(\g)=(F(\g), r(\g))$ for $\g\in\mg$. We observe that $\eta$ preserves the  composition and the  inverse. It is evident that $\eta$ is a bijection and preserves the grading of groupoids. Thus $\eta$ is an isomorphism of $G$-graded groupoids. By Corollary \ref{corvon}, in this case $A_K(\mg)\cong_{\gr}A_K(\mg_X)$ is $G$-graded von Neumann regular.

Specialising this to the groupoid $\mg_E$ of a graph $E$ with a star injective functor $\widetilde{w}:\mg_E\xra \Z$, we have a partial action $\phi$ of $\Z$ on $\mg_E^{(0)}$ and realise the Leavitt path algebra $L_R(E)$ as the partial skew group ring $C_R(X)\rtimes_{\phi}\Z$.

\begin{cor}\label{kkkkkk} Let $E$ be a graph which satisfies $|r^{-1}_E(v)|\leq 1$ for each $v\in E^0$ and $w: E^1\xra \Z$ a function assigning each edge to $1$. Then the Leavitt path algebra $L_R(E)$ is $\Z$-graded isomorphic to the partial skew group ring $C_R(X)\rtimes_{\phi}\Z$.
\end{cor}
\begin{proof}
By Lemma \ref{iffstarinj}, $\widetilde{w}:\mg_E\xra \Z$ is star injective. Then the consequence follows from \eqref{isogroupoid} and Proposition \ref{propss}.
\end{proof}

As an example, the Corollary~\ref{kkkkkk} shows that the Toepliz algebra can be written as a partial skew integral group ring.

%\note{ $G_E$ is star injective if $E$ is a cycle or infinite path, only in this case $L(E)$ can be written as $C(X) \rtimes \mathbb Z$}

%\note{The above statement is not correct, $G$ here has to be the free group otherwise how do you define the isomorphisms of the groupoid. Also $\mathbb Z$ is sitting in this description but we know we can't write LPA as partrial skew by $\mathbb Z$}

\section{Acknowledgements}
The authors would like to acknowledge Australian Research Council grants DP150101598 and
DP160101481.

\end{document}